\documentclass[11pt,letterpaper]{amsart}
\usepackage{amssymb,latexsym}
\usepackage{amsfonts}
\usepackage{amsmath}
\usepackage{enumitem}
\usepackage{marvosym}
\usepackage{lipsum}

\newcommand{\bp}{\overline{\partial}}

\newcommand{\vps}{\varepsilon}

\newcommand{\og}{\omega}

\newcommand{\lraw}{\longrightarrow}

\newcommand{\lag}{\langle}
\newcommand{\rag}{\rangle}

\newtheorem{thm}{Theorem}[section]
\newtheorem{cor}[thm]{Corollary}
\newtheorem{lem}[thm]{Lemma}
\newtheorem{prop}[thm]{Proposition}
\newtheorem{defn}[thm]{Definition}
\newtheorem{rem}[thm]{\bf{Remark}}

\numberwithin{equation}{section}

\DeclareMathOperator \Vol{Vol}
\DeclareMathOperator \Id{Id}
\DeclareMathOperator \rank{rank}

\DeclareMathOperator \End{End}

\DeclareMathOperator \tr{tr}

\def \pbp {\partial\bar\partial}

\begin{document}

\vspace*{0.5cm} \begin{center}\noindent {\LARGE \bf Mean curvature positivity and rational connectedness  }\\[0.7cm]
Chao Li\footnotemark[1]
\\[0.2cm]
School of Mathematics and Statistics\\
Nanjing University of Science and Technology\\
Nanjing, 210094, P.R.China\\
E-mail: leecryst@mail.ustc.edu.cn\\[0.5cm]
Chuanjing Zhang\footnotemark[1]
\\[0.2cm]
School of Mathematics and Statistics\\
Ningbo University\\ Ningbo, 315021, P.R. China\\
E-mail: zhangchuanjing@nbu.edu.cn\\[0.5cm]
Xi Zhang\footnotemark[1]
\\[0.2cm]
School of Mathematics and Statistics\\
Nanjing University of Science and Technology\\
Nanjing, 210094, P.R.China\\
E-mail: mathzx@njust.edu.cn\\[1cm]
\footnotetext[1]{\notag\noindent The research was supported by the National Key R and D Program of China 2020YFA0713100. The  authors are partially supported by NSF in China No.12141104 and 12371062, the China Postdoctoral Science Foundation (No.2018M640583, 2018M642515) and the Fundamental Research Funds for the Central Universities.}

\end{center}


{\bf Abstract.} In this paper, we use Uhlenbeck-Yau's continuity method to establish the correspondence between the mean curvature positivity and the HN-positivity on holomorphic vector bundles over compact Hermitian manifolds. As its application, we get a differential geometric criterion for rational connectedness, i.e. we prove that a compact K\"ahler manifold is projective and rationally connected if and only if its holomorphic tangent bundle is mean curvature positive.
\\[1cm]
{\bf AMS Mathematics Subject Classification.} 53C07, 14J60, 32Q15 \\
 {\bf
Keywords and phrases.} Holomorphic vector bundles,\ Hermitian manifold,\ mean curvature,\ rational
connectedness,\ holomorphic map


\newpage

\section{ Introduction}

In the study of stabilities and vanishing theorems  of holomorphic vector bundles (\cite{DON85,KobW70,NS65,UY86}), slope and mean curvature are very important notions. Actually, they are also extremely useful in the study of positivities of vector bundles. For a holomorphic vector bundle over a compact Riemannian surface (\cite{CF,Um}), ampleness is equivalent to slope positivity (i.e. the positivity of the minimum of  the slopes of quotient bundles), and also to mean curvature positive (i.e. the existence of a Hermitian metric with positive mean curvature). In higher dimension case, the situation is quite different. The motivation of this paper is to study the equivalence relationship between slope positivity and mean curvature positivity, and related problems.

\begin{defn}
Let $(E, \bar{\partial }_{E})$ be a holomorphic vector bundle over a complex manifold $M$ of dimension $n$. Given a Hermitian metric $\omega$ on $M$ and a Hermitian metric $H$ on $E$, we call $\sqrt{-1}\Lambda_{\omega }F_{H}$ the $\omega$-mean curvature of $H$, where $F_{H}$ is the curvature form of Chern connection $D_{H}$ with respect to the Hermitian metric $H$, $\Lambda_{\omega }$ denotes the contraction with  $\omega $.
\begin{enumerate}
\item  $(E, \bar{\partial }_{E})$ is called mean curvature positive (resp. nonnegative) if there is a Hermitian metric $\hat{\omega } $ on $M$ and a Hermitian metric $H$ on $E$ such that $\sqrt{-1}\Lambda_{\hat{\omega } }F_{H}>0$ (resp. $\sqrt{-1}\Lambda_{\hat{\omega }}F_{H}\geq 0$ ).
\item $(E, \bar{\partial }_{E})$ is called uniformly mean curvature positive (resp. nonnegative) if for any Hermitian metric $\omega$ on $M$, there exists a Hermitian metric $H$ on $E$ with $\sqrt{-1}\Lambda_{\omega}F_H>0$ (resp. $\geq 0$).
\end{enumerate}

We say $(E, \bar{\partial }_{E})$ is mean curvature negative (resp. nonpositive) if its dual bundle is mean curvature positive (resp. nonnegative). A  complex manifold $M$ is called mean curvature positive (resp. nonnegative) if its holomorphic tangent bundle $T^{1,0}M$ is mean curvature positive (resp. nonnegative).

\end{defn}

We say a Hermitian metric $\omega $  is Gauduchon if it satisfies $\partial \bar{\partial} \omega^{n-1}=0$. Gauduchon (\cite{Gaud}) proved that on a compact complex manifold, there is a unique Gauduchon metric $\omega $ up to a positive constant in the conformal class of every Hermitian metric $\hat{\omega }$.
Assume $(M, \omega ) $ is a compact Gauduchon manifold and $\mathcal{F}$ is a coherent sheaf over $M$. The $\omega $-degree of $\mathcal{F}$ is given by
\begin{equation}
\deg_{\omega }(\mathcal{F} ):=\deg_{\omega }(\det(\mathcal{F}) )=\int_{M}c_{1}(\det(\mathcal{F}), H)\wedge \frac{\omega^{n-1}}{(n-1)!},
\end{equation}
where $H$ is an arbitrary Hermitian metric on $\det{\mathcal{F}}$. This is a well-defined real number independent of $H$ since $\omega^{n-1}$ is $\partial \bar{\partial }$-closed. We define the $\omega$-slope of $\mathcal{F}$ as
\begin{equation}
\mu_{\omega }(\mathcal{F} ):=\frac{\deg_{\omega }(\mathcal{F} )}{\rank (\mathcal{F} )}.
\end{equation}
 A holomorphic vector bundle $(E , \bar{\partial }_{E})$ is called $\omega $-stable (semistable) if for every proper saturated subsheaf $\mathcal{S}\subset E$, there holds
\begin{equation}
\mu_{\omega}(\mathcal{S})<(\leq)\mu_{\omega}(E).
\end{equation}
We say $H$  is
a Hermitian-Einstein metric on $(E, \bar{\partial }_{E} )$ if it satisfies
\begin{equation}
\sqrt{-1}\Lambda_{\omega} F_{H} =\lambda\cdot \mathrm{Id}_{E},
\end{equation}
where  $\lambda =\frac{2\pi }{\Vol(M, \omega )}\mu_{\omega}(E) $. In this paper we always fix the holomorphic structure $\bar{\partial }_{E}$, and for brevity we sometimes denote $(E, \bar{\partial }_{E} )$ by $E$. The classical Donaldson-Uhlenbeck-Yau theorem (\cite{NS65,DON85,UY86}) states that, when $\omega$ is K\"ahler, the stability implies the existence of  Hermitian-Einstein metric. According to \cite{Bu,LT}, we know that the Donaldson-Uhlenbeck-Yau theorem is also valid for compact Gauduchon manifolds. There
are  many other interesting and important works related (\cite{ag,BS,Bi,bis0,bis,br,DW, HIT, HuLe,Jacob1,Kobayashi,JZ, LN2, LT,LT95, LZ, LZZ, LZZ2, Mo1, Mo2, Mo3,NR01,SIM,SIM2,WZ}, etc.). In \cite{NZ}, Nie and the third author proved that on a compact Gauduchon manifold $(M, \omega)$, every semistable holomorphic vector bundle $(E, \bar{\partial }_{E})$  admits an approximate Hermitian-Einstein structure, i.e. for any $\delta >0$, there exists a Hermitian metric $H_{\delta}$ such that
\begin{equation}
\sup_M|\sqrt{-1}\Lambda_{\omega }F_{H_{\delta}}-\lambda \cdot \textmd{Id}_E|_{H_{\delta}}<\delta.
\end{equation}
This means that every semistable holomorphic vector bundle $(E, \bar{\partial }_{E})$ over the compact Gauduchon manifold $(M, \omega)$ must admit a Hermitian metric with negative mean curvature if $\mu_{\omega }(E)<0$.

Let $\mathcal{S}$ be a coherent subsheaf of the holomorphic vector bundle $(E, \bar{\partial }_{E})$, and $H$ be a Hermitian metric on $E$. Bruasse (\cite{Bru}) derived the following Chern-Weil formula:
\begin{equation}\label{B1}
\begin{split}
\deg_{\omega }(\mathcal{S} )=&\int_{M\setminus \Sigma_{alg}}\frac{\sqrt{-1}}{2\pi}\tr F_{H_{\mathcal{S}}}\wedge \frac{\omega^{n-1}}{(n-1)!}\\
=& \frac{1}{2\pi}\int_{M\setminus \Sigma_{alg}} (\sqrt{-1}\textmd{tr}(\pi_{\mathcal{S}}^{H} \Lambda_{\omega} F_{H})-|\bar{\partial }\pi_{\mathcal{S}}^{H} |_{H}^{2})\frac{\omega^{n}}{n!},
\end{split}
\end{equation}
where $\Sigma_{alg}$ is the singular set of $\mathcal{S}$, $H_{\mathcal{S}}$ is the induced metric on $\mathcal{S}|_{M\setminus \Sigma_{alg}}$ and $\pi_{\mathcal{S}}^{H}$ is the orthogonal projection onto $\mathcal{S}$ with respect to the metric $H$.
We know that $\deg_{\omega }(S )$ is bounded from above.
Bruasse (\cite{Bru}) also proved that one can find a maximal subsheaf which realizes the supremum of the slopes, i.e. there exists a coherent subsheaf $\mathcal{F}$ such that
\begin{equation}\label{B2}
\mu_{\omega }(\mathcal{F} )=\mu_{U}(E, \omega ):=\sup\{\mu_{\omega}(\mathcal{S}) \; \big| \; \mathcal{S}\textrm{ is a coherent subsheaf of $E$} \}.
\end{equation}
Then it can be seen that the infimum of the slopes of coherent quotient sheaves can be attained, i.e. there exists a coherent quotient sheaf $\bar{\mathcal{Q}}$ such that
\begin{equation}\label{B21}
\mu_{\omega }(\bar{\mathcal{Q}} )=\mu_{L}(E, \omega ):=\inf\{\mu_{\omega}(\mathcal{Q}) \; \big| \; \mathcal{Q}\textrm{ is a coherent quotient sheaf of $E$} \}.
\end{equation}
Furthermore, there is a unique filtration of $(E, \bar{\partial}_{E})$ by subsheaves
\begin{equation}\label{HNS01}
0=\mathcal{E}_{0}\subset \mathcal{E}_{1}\subset \cdots \subset \mathcal{E}_{l}=E
\end{equation}
such that every quotient sheaf $\mathcal{Q}_{\alpha }=\mathcal{E}_{\alpha }/\mathcal{E}_{\alpha -1}$ is torsion-free and $\omega$-semistable, and $\mu_{\omega } (\mathcal{Q}_{\alpha })>\mu_{\omega } (\mathcal{Q}_{\alpha +1})$,  which is called the
Harder-Narasimhan filtration of $(E, \bar{\partial}_{E})$. If $\rank(E)=r$, we have a nonincreasing
$r$-tuple of numbers
\begin{equation}
\vec{\mu}_\omega(E )=(\mu_{1, \omega} , \cdots , \mu_{r, \omega})
\end{equation}
from the HN-filtration by setting: $\mu_{i, \omega }=\mu_{\omega} (\mathcal{Q}_{\alpha})$, for
$\rank(\mathcal{E}_{\alpha -1})+1\leq i\leq \rank(\mathcal{E}_{\alpha})$. We call $\vec{\mu}_\omega(E)$ the Harder-Narasimhan type of $(E, \bar{\partial}_{E})$. It is easy to see that
\begin{equation}
\mu_{1, \omega }=\mu_{U}(E, \omega ) \quad \text{and} \quad \mu_{r, \omega }=\mu_{L}(E, \omega ).
\end{equation}
For each $\mathcal{E}_{\alpha }$ and the Hermitian metric $K$,  we have the associated orthogonal
projection $\pi_{\alpha }^{K}:E\rightarrow E$ onto $\mathcal{E}_{\alpha }$ with respect to  $K$. It is well  known that every
$\pi_{\alpha }^{K}$ is an $L_{1}^{2}$-bounded  Hermitian endomorphism. We define an $L_{1}^{2}$-bounded Hermitian endomorphism by
\begin{equation}\Phi_{\omega}^{HN} (E, K)=\Sigma_{\alpha=1}^{l}\mu_{\omega } (\mathcal{Q}_{\alpha })(\pi_{\alpha }^{K}-\pi_{\alpha-1}^{K}),\end{equation} which will be called the
Harder-Narasimhan projection of  $(E, \bar{\partial}_{E})$.

\begin{defn}
Let $(E, \bar{\partial }_{E})$ be a holomorphic vector bundle over a compact complex manifold $M$.
\begin{enumerate}
\item Given a Gauduchon metric $\omega $ on $M$, we say $(E, \bar{\partial }_{E})$ is $\omega $-HN-positive (resp. $\omega $-HN-nonnegative)  if $\mu_{L}(E, \omega )>0$ (resp. $\mu_{L}(E, \omega )\geq 0$).

\item $(E, \bar{\partial }_{E})$ is called HN-positive (resp. HN-nonnegative) if there is a Gauduchon metric $\omega $ on the base manifold $M$ such that $\mu_{L}(E, \omega )>0$ (resp. $\mu_{L}(E, \omega )\geq 0$).

\end{enumerate}

    We say $(E, \bar{\partial }_{E})$ is HN-negative (resp. HN-nonpositive) if its dual bundle is HN-positive (resp. HN-nonnegative).
A compact complex manifold $M$ is called HN-positive (resp. HN-nonnegative, HN-negative, HN-nonpositive) if its holomorphic tangent bundle $T^{1,0}M$ is HN-positive (resp. HN-nonnegative, HN-negative, HN-nonpositive).
\end{defn}

Campana and P\u{a}un (\cite{CP1,CP2}) have studied the $\alpha$-slope positivity for some movable class $\alpha$ on the projective manifold. In this paper, we discuss the general complex manifold case, and establish the equivalence between mean curvature positivity (resp.  negativity) and HN-positivity (resp.  HN-negativity).
Let $(M, \hat{\omega } )$ be a compact Hermitian manifold, $\omega $ be a Gauduchon metric in the conformal class of $\hat{\omega }$. If the mean curvature $\sqrt{-1}\Lambda_{\hat{\omega } }F_{H} >0$, equivalently $\sqrt{-1}\Lambda_{\omega  }F_{H} >0$,  by the formula (\ref{B1}), we know that
$\deg_{\omega}(E)-\deg_{\omega}(\mathcal S)>0$ for any coherent subsheaf $\mathcal{S}$. Then
\begin{equation}
\mu_{L}(E,\omega)>0,
\end{equation}
i.e. $(E, \bar{\partial }_{E})$ is $\omega$-HN-positive. To establish the above equivalence, we only need to prove that $\omega$-HN-positivity implies $\omega$-mean curvature positivity, i.e. there exists a Hermitian metric $H$ on $E$ such that $\sqrt{-1}\Lambda_{\omega  }F_{H} >0$.

We denote the $r$ eigenvalues of the mean curvature $\sqrt{-1}\Lambda_{\omega} F_H$  by $\lambda_1(H,\omega)$, $\lambda_2(H,\omega)$, $\cdots$, $\lambda_r(H,\omega)$, sorted in the descending order. Then each $\lambda_{\alpha }(H,\omega)$ is Lipschitz continuous.
Set
\begin{equation}\vec\lambda(H,\omega)=(\lambda_1(H,\omega),\lambda_2(H,\omega),\cdots,\lambda_r(H,\omega)),
\end{equation}
\begin{equation}
\lambda_L(H,\omega)=\lambda_r(H,\omega),\quad \lambda_U(H,\omega)=\lambda_1(H,\omega),
\end{equation}
\begin{equation}\hat{\lambda }_L(H,\omega)=\inf_M\ \lambda_L(H,\omega),\quad \hat{\lambda }_U(H,\omega)=\sup_M\ \lambda_U(H,\omega)
\end{equation}
and
\begin{equation}
\begin{split}
&\lambda_{mL}(H,\omega)=\frac{1}{\Vol(M, \omega)}\int_M\lambda_L(H,\omega)\frac{\omega^n}{n!},\\&\lambda_{mU}(H,\omega)=\frac{1}{\Vol(M, \omega)}\int_M\lambda_U(H,\omega)\frac{\omega^n}{n!}.
\end{split}
\end{equation}

In this paper, we first study the following perturbed Hermitian-Einstein equation on $(E, \bar{\partial }_{E})$:
\begin{equation} \label{eq}
\sqrt{-1}\Lambda_{\omega } F_{H}-\lambda \cdot \textmd{Id}_E+\varepsilon \log (K^{-1}H)=0,
\end{equation}
where $K$ is any fixed background metric. Uhlenbeck and Yau first introduced the above perturbed equation in \cite{UY86}, where they used the continuity method to prove the Donaldson-Uhlenbeck-Yau theorem. Due to the fact that the elliptic operators are Fredholm if the base manifold is compact, the equation (\ref{eq}) can be solved for any $\varepsilon \in (0, 1]$. Let $H_{\varepsilon}$ be a solution of perturbed equation (\ref{eq}). When $(M, \omega)$ is a compact Gauduchon manifold, Nie and the third author (\cite[Proposition 3.1]{NZ}) have the key observation:
\begin{equation}\label{eq33}
\int_M \big(\tr((\sqrt{-1}\Lambda_{\omega }F_{K}-\lambda \cdot \textmd{Id}_E) s_{\varepsilon})+\lag\bar{\Psi }(s_{\varepsilon})(\bar{\partial}s_{\varepsilon}), \bar{\partial}s_{\varepsilon}\rag_{K}\big)\frac{\og^n}{n!}=-\varepsilon\int_M \tr(s_{\varepsilon}^{2})\frac{\og^n}{n!},
\end{equation}
where $s_{\varepsilon}=\log (K^{-1}H_{\varepsilon})$ and
\begin{equation}\label{eq3301}
\bar{\Psi}(x, y)=\left\{
\begin{split}
&\frac{e^{y-x}-1}{y-x}, & x\neq y;\\
&\ \ \ \  1, & x=y.
\end{split}\right.
\end{equation}
By using the above identity (\ref{eq33}) and Uhlenbeck-Yau's result (\cite{UY86}) that $L_{1}^{2}$
weakly holomorphic subbundles define saturated coherent subsheaves, following Simpson's argument in \cite{SIM}, we can obtain the existence of
$L^{p}$-approximate critical Hermitian structure on $(E, \bar{\partial }_{E})$, i.e. we proved the following theorem.

\begin{thm} \label{theorem1}
Let $(M, \omega )$ be a compact Gauduchon manifold of complex dimension $n$,  $(E,\bar{\partial}_E)$ be a  holomorphic vector bundle of rank $r$ over $M$, $K$ be a fixed Hermitian metric on $E$ and $H_{\varepsilon}$ be a solution of perturbed equation (\ref{eq}).  Then there exists a sequence $\varepsilon_{i}\rightarrow 0$ such that
 \begin{equation}\label{eqthm1}
 \lim_{i\rightarrow \infty }\left\| \sqrt{-1}\Lambda_{\omega }F_{H_{\varepsilon_{i}}}-\frac{2\pi }{\Vol(M, \omega )}\Phi_{\omega }^{HN}(E, K)\right\|_{L^{p}(K)}=0
 \end{equation}
 for any $0<p<+\infty $. In particular,
 \begin{equation}
 \lim_{i\rightarrow \infty}\lambda_{mL}(H_{\varepsilon_{i}}, \omega )=\frac{2\pi}{\Vol(M,\omega)}\mu_{r, \omega}=\frac{2\pi}{\Vol(M,\omega)}\mu_{L}(E, \omega )
 \end{equation}
 and
 \begin{equation}
 \lim_{i\rightarrow \infty}\lambda_{mU}(H_{\varepsilon_{i}}, \omega )=\frac{2\pi}{\Vol(M,\omega)}\mu_{1, \omega}= \frac{2\pi}{\Vol(M,\omega)}\mu_{U}(E, \omega ).
 \end{equation}
\end{thm}

\medskip

In the K\"ahler case, the $L^{p}$-approximate critical Hermitian structure was first suggested by Daskalopoulos and Wentworth (\cite{DW}), and its existence plays a crucial role in proving the Atiyah-Bott-Bando-Siu conjecture (\cite{DW,Ja3,Sib,LZZ2}). It should be pointed out that, even in the K\"ahler case,  our proof is new and very different from previous proofs, where they depend on the resolution of singularities theorem of Hironaka (\cite{Hi2}) and the Donaldson-Uhlenbeck-Yau theorem of reflexive sheaf by Bando and Siu (\cite{BS}).

According to  Theorem \ref{theorem1}, for any $\delta >0$, there exists a Hermitian metric $H_{\varepsilon }$ satisfying $\lambda_{mL}(H_{\varepsilon}, \omega)>\frac{2\pi}{\Vol(M,\omega)}\mu_{L}(E, \omega )-\delta $. Let $\tilde \lambda$ be a smooth function such that
$$\tilde \lambda\leq \lambda_L(H_{\varepsilon},\omega)+\lambda_{mL}(H_{\varepsilon}, \omega)-\frac{2\pi}{\Vol(M,\omega)}\mu_{L}(E, \omega )+\delta$$
and $\int_M\tilde{\lambda}\frac{\omega^n}{n!}=\lambda_{mL}(H_\varepsilon,\omega)\Vol(M,\omega)$. Since $\omega $ is a Gauduchon metric, by a conformal transformation as in \cite{Kobayashi}, taking $\hat{H}_{\delta}=e^{f}H_{\varepsilon}$ where $f$ is defined by $\sqrt{-1}\Lambda_{\omega }\partial \bar{\partial }f=\tilde \lambda-\lambda_{mL}(H_{\varepsilon},\omega)$, we have
 \begin{equation}
 \sqrt{-1}\Lambda_{\omega }F_{\hat{H}_{\delta}}\geq \frac{2\pi}{\Vol(M,\omega)}\mu_{L}(E, \omega )-\delta .
 \end{equation}
 So we establish the following equivalence between mean curvature positivity (resp. negativity) and HN-positivity (resp. HN-negativity).

\begin{thm}\label{cor0}
Let $(M, \tilde{\omega} )$ be a compact Hermitian manifold, and $(E, \bar{\partial }_{E})$ be a holomorphic vector bundle over $M$. Then there exists a Hermitian metric $H$ such that the mean curvature $\sqrt{-1}\Lambda_{\tilde{\omega}}F_{H}$ is positive (resp.  negative ) if and only if  $\mu_{L}(E, \omega )>0$ (resp.  $\mu_{U}(E, \omega )<0$), where $\omega$ is a Gauduchon metric conformal to $\tilde{\omega}$.
\end{thm}


Actually, we can easily deduce the following characterization of the minimal and maximal slopes in the Harder-Narasimhan type.
\begin{thm}\label{minslope}
Let $(M, \omega )$ be a compact Gauduchon manifold,  $(E,\bar{\partial}_E)$ be a  holomorphic vector bundle over $M$. Then $\frac{2\pi}{\Vol(M,\omega)}\mu_{L}(E, \omega )$ is equal to
\begin{equation}
\sup\{t|\text{There is a Hermitian metric $H$ with $\sqrt{-1}\Lambda_{\omega }F_{H}\geq t\Id_E$}\},
\end{equation}
and $\frac{2\pi}{\Vol(M,\omega)}\mu_{U}(E, \omega )$ is equal to
\begin{equation}
\inf\{t|\text{There is a Hermitian metric $H$ with $\sqrt{-1}\Lambda_{\omega }F_{H}\leq t\Id_E$}\}.
\end{equation}
\end{thm}

\begin{rem}\label{rem001}

 By Theorem \ref{cor0}, it is easy to see that the following statements on $E$ are equivalent:
\begin{enumerate}
\item  [\rm{(1)}]   $(E, \bar{\partial }_{E})$ is HN-positive (resp.  HN-negative);
\item  [\rm{(2)}] there is a Hermitian metric $\omega $ on the base manifold $M$ and a Hermitian metric $H$ on $(E, \bar{\partial }_{E})$ such that the mean curvature $\sqrt{-1}\Lambda_{\omega} F_{H}$ is positive (resp.  negative);
\item  [\rm{(3)}] there is a Hermitian metric $\omega $ on the base manifold $M$ and a Hermitian metric $H$ on $(E, \bar{\partial }_{E})$ such that the mean curvature $\sqrt{-1}\Lambda_{\omega} F_{H}$ is quasi-positive (resp.  quasi-negative);
\item [\rm{(4)}]  there is a Gauduchon metric $\omega $ on the base manifold $M$ and a Hermitian metric $H$ on $(E, \bar{\partial }_{E})$ such that $\lambda_{mL}(H,\omega)>0$ (resp.  $\lambda_{mU}(H,\omega)<0$).
\end{enumerate}

\end{rem}

A projective manifold $M$ is called rationally connected if any two points on $M$ can be connected by some rational curves. The rational connectedness is an important concept in algebraic geometry, and many people have given the criteria for it (\cite{CDP14,GHS,Pet,LP17,Cam16}). It is of interest to give a geometric interpretation for rational connectedness.
By using Campana-Demailly-Peternell's criterion for rational connectedness (\cite[Criterion 1.1]{CDP14}), Yang (\cite[Corollary 1.5]{Ya0}) proved that  for a compact K\"ahler manifold, if its holomorphic tangent bundle is mean curvature positive, then it must be projective and rationally connected.  This also confirmed the well-known Yau's conjecture (\cite[Problem 47]{Yau82}) that the compact K\"ahler manifold with positive holomorphic sectional
curvatures must be projective and rationally connected.
We naturally ask the reverse question: if $M$ is projective and rationally connected, is its holomorphic tangent bundle mean curvature positive?

 As an application of Theorem \ref{cor0} and Campana-Demailly-Peternell's result (\cite[Criterion 1.1]{CDP14}), we show that rational connectedness implies mean curvature positivity. This also solves a question of Demailly and Yang (Problem 4.17 in \cite{Ya1}). Thus we arrive at the following theorem.

\begin{thm}\label{RC}
Let $M$ be a compact K\"ahler manifold. Then $M$ is projective and rationally connected if and only if its holomorphic tangent bundle $T^{1,0}M$ is mean curvature positive, i.e. there exist a Hermitian metric $\omega $ on $M$ and a Hermitian metric $H$ on $T^{1,0}M$ such that $\sqrt{-1}\Lambda_{\omega }F_{H}>0$.
\end{thm}

In the last part of this paper, we will present some applications of Theorem \ref{minslope} (and Theorem \ref{cor0}).

\begin{defn}
Let $(E, \bar{\partial}_{E})$ and  $(\tilde{E}, \bar{\partial}_{\tilde{E}})$ be two holomorphic vector bundles over a compact complex manifold $M$. If $(E, \bar{\partial }_{E})$ is HN-negative, we define
\begin{equation}
G(M, E, \tilde{E}):=\inf_{\omega \in \mathcal{G}^{+}(M, E) }\frac{-\mu_U(\tilde{E}, \omega )}{\mu_U(E, \omega )},
\end{equation}
where $\mathcal{G}^{+}(M, E)$ denotes the space of Gauduchon metric $\omega $ on $M$ such that $(E, \bar{\partial }_{E})$ is $\omega$-HN-negative.
\end{defn}

We derive the following vanishing theorem.

\begin{thm}\label{cor111}
Let $(E, \bar{\partial}_{E})$ and  $(\tilde{E}, \bar{\partial}_{\tilde{E}})$ be two holomorphic vector bundles over a compact complex manifold $M$.
If $(E, \bar{\partial }_{E})$ is HN-negative, then
    \begin{equation}
    H^{0}(M, E^{\otimes k}\otimes \tilde{E}^{\otimes l})=0
    \end{equation}
    when $k\geq 1$, $l\geq 0$ and $k>G(M, E, \tilde{E})l$.

    \end{thm}

A holomorphic vector bundle $(E, \bar{\partial }_{E})$ is said to be ample if its tautological line bundle $\mathcal{O}_{E}(1)$ is ample over the projective bundle $\mathbb P(E)$ of hyperplanes of $E$. The notion of positivity is very important in both algebraic geometry and complex geometry. In \cite{Gri69}, Griffiths introduced the following positivity: For a Hermitian metric $H$ on $E$, an $H$-Hermitian $(1,1)$-form $\sqrt{-1}\Theta$ valued in $\End E$ is said to be Griffiths positive, if
at every $p\in M$,  it holds that
\begin{equation}
\langle \Theta(v,\bar v) u,u \rangle_H>0
\end{equation}
for any non-zero vector $u\in E|_p$ and any non-zero vector $v\in T^{1,0}_pM$. We say $(E, \bar{\partial }_{E})$ is Griffiths positive if $E$ admits a Hermitian metric $H$ such that $\sqrt{-1} F_H$ is Griffiths positive.
Of course a Griffiths  positive vector bundle is
ample. However, it is still an open problem of Griffiths (\cite{Gri69}) that the ampleness implies Griffiths-positivity. In \cite{Gri69}, Griffiths raised the
question to determine which characteristic forms are positive on Griffiths positive vector bundles, and proved the second Chern form is positive for the rank two case. Recently,  there are several interesting works on the above Griffiths' question (\cite{Mou,Gul,Div,Pinga,Li20,Fin20,Xiao,DF}). As an application of Theorem \ref{cor0}, we deduce


\begin{thm}\label{ample}

Let $(E, \bar{\partial}_{E})$ be an ample holomorphic vector bundle over a compact complex manifold $M$. Then $(E, \bar{\partial}_{E})$ is uniformly mean curvature positive, i.e. for any Hermitian metric $\omega $ on $M$, there exists a Hermitian metric $H$ on $E$ such that $\sqrt{-1}\Lambda_{\omega }F_{H}>0$.
\end{thm}

\begin{defn}
Let $(E, \bar{\partial }_{E})$ be a  holomorphic vector bundle over a complex manifold $M$. A Hermitian metric $H$ on $E$ is called RC-positive at point $p\in M$ if  for  any non-zero vector $e$ of $E|_{p}$, there exists a vector $v$ of $T_{p}^{1,0}M$ such that $\langle F_{H}(v, \bar{v})e, e\rangle_{H}>0$. The Hermitian metric $H$ is called uniformly RC-positive at point $p\in M$ if  there exists a vector $v$ of $T_{p}^{1,0}M$ such that $\langle F_{H}(v, \bar{v})e, e\rangle_{H}>0$ for  any non-zero vector $e$ of $E|_{p}$. The Hermitian metric $H$ on $E$ is called RC-positive (resp. uniformly RC-positive) if it is RC-positive (resp. uniformly RC-positive) at all points of $M$. We say $(E, \bar{\partial }_{E})$ is RC-positive (resp. uniformly RC-positive) if it admits a Hermitian metric $H$ which is RC-positive (resp. uniformly RC-positive).
\end{defn}

\begin{rem}
	The concepts of RC-positivity and uniformly RC-positivity were introduced by Yang in \cite{Ya0,Ya1}, and they are very effective in studying the vanishing theorems. From the definition, one can easily see that  a Hermitian metric $H$ with positive mean curvature  must be RC-positive. On the other hand, it's not hard to prove that if $H$ is uniformly RC-positive, then it must be mean curvature positive, i.e. $\sqrt{-1}\Lambda_{\omega }F_{H}>0$ with respect to some Hermitian metric $\omega $ on $M$ (see Proposition \ref{RC01} for details). Hence the mean curvature positivity is an intermediate concept between uniformly RC-positivity and  RC-positivity. If $(E, \bar{\partial }_{E})$ is ample, by virtue of Theorem \ref{ample}, there exists a Hermitian metric $H$ with positive mean curvature on $E$. So for every $1\leq s \leq \rank(E)$ (resp. every $k\geq 1$), the induced metric $\wedge^{s}H$ (resp. $\otimes ^{k}H$) on the bundle $\wedge^{s}E$ (resp. $\otimes ^{k}E$) must have positive mean curvature, and then is also RC-positive. This confirms a conjecture proposed by Yang (\cite[Conjecture 7.10]{Ya0}).
\end{rem}

The holomorphic map is an important research object in complex geometry. There are many generalizations of the classical Schwarz Lemma and rigidity result on holomorphic maps via the works of  Ahlfors, Chern, Lu, Yau and others (\cite{Ah,Chern,Lu,Yau,Roy,YC,Ni,Ya,Zy}). As another application of Theorem \ref{minslope}, we obtain the following integral inequality for holomorphic maps.

\begin{thm}\label{holomorphic1}
Let $f$ be a holomorphic map from a compact Gauduchon manifold $(M, \omega )$ to a Hermitian manifold $(N, \nu )$.  If $f$ is not constant, then there holds
\begin{equation}
2\pi \mu_{L}(T^{1, 0}M, \omega )\leq \int_{M}HB^{\nu }_{f(\cdot)}\cdot f^{\ast }(\nu ) \wedge \frac{\omega^{m-1}}{(m-1)!},
\end{equation}
where $m=\dim ^{\mathbb{C}}M$ and $HB^{\nu}_{f(x)}$ is the supremum of holomorphic bisectional curvatures at $f(x)\in (N, \nu)$.
\end{thm}

Therefore, we conclude the following rigidity result of holomorphic maps.

\begin{cor}\label{holomorphic3}
Let $f$ be a holomorphic map from a compact complex manifold $M$ to a complex manifold $N$. If $M$ is HN-nonnegative ( resp. HN-positive) and $N$ admits a Hermitian metric with negative holomorphic bisectional curvature (resp. nonpositive holomorphic bisectional curvature), then $f$ must be constant.
\end{cor}


This paper is organized as follows.  Section 2 is devoted to the proof of Theorem \ref{theorem1}. In Section 3, we show Theorem \ref{RC}.  In Section 4, we prove Theorem \ref{cor111}, Theorem \ref{ample}  and Theorem \ref{holomorphic1}.

\medskip



\section{The existence of $L^{p}$-approximate critical Hermitian structure}

In this section we give a proof of Theorem \ref{theorem1}. Let $(M, \omega )$ be a compact Gauduchon manifold of complex dimension $n$ and $(E, \bar{\partial}_E)$ a rank $r$ holomorphic vector bundle endowed  with a Hermitian metric $K$ over $M$. Without loss of generality, we can always  assume $\tr (\sqrt{-1}\Lambda_{\omega } F_{K}-\lambda  \textmd{Id}_E)=0$ with $\lambda =\frac{2\pi }{\Vol(M, \omega )}\mu_{\omega}(E)$. For the convenience of the reader, we explain the main steps of our proof. By \cite{NZ}, we only need to consider the non $\omega$-semistable case. Let $H_{\varepsilon}$ be a solution of perturbed equation (\ref{eq}), and set $h_{\varepsilon}=K^{-1}H_{\varepsilon}$, $s_{\varepsilon}= \log h_{\varepsilon}$, $l_{\varepsilon}= \| s_{\varepsilon}\|_{L^2}$, $u_{\varepsilon}= \frac{s_{\varepsilon}}{l_{\varepsilon}}$. By using the identity (\ref{eq33}) and the arguments of Simpson \cite{SIM},  we can show that, by choosing a subsequence, $u_{\varepsilon} \rightharpoonup u_\infty$ weakly in $L_1^2$, the eigenvalues $\{\mu_A\}_{A=1}^{l}$ of $u_{\infty}$ are constants and $A\geq 2$. According to the regularity statement for $L_1^2$-subbundles in \cite{UY86}, we can construct a saturated subsheaf $E_A$ of $E$ with respect to every distinct eigenvalue $\mu_A$ of $u_{\infty}$, and obtain the following filtration of $(E, \bar{\partial}_E)$
 \begin{equation}
0= E_0\subset E_1 \subset E_2 \subset\cdots \subset E_l=E.
\end{equation}
Furthermore, we prove that this filtration is exactly the Harder-Narasimhan filtration of $(E, \bar{\partial}_E)$, and obtain the existence of the $L^{p}$-approximate critical Hermitian structure.

First we review some of the standard facts on the perturbed Hermitian-Einstein equation (\ref{eq}).
For any Hermitian metric  $H$ on $E$, we denote  the Chern connection by $D_{H}$, the $(1,0)$-part of $D_H$ by $\partial_H$ and the curvature form by $F_{H}$. Set $h=K^{-1}H$, then we have the following identities
\begin{equation}\label{id1}
\begin{split}
& \partial _{H}-\partial_{K} =h^{-1}\partial_{K}h,\\
& F_{H}-F_{K}=\bar{\partial }_{E} (h^{-1}\partial_{K} h) .
\end{split}
\end{equation}
As a consequence, the equation (\ref{eq}) can be rewritten as
\begin{equation} \label{eq2}
\sqrt{-1}\Lambda_{\omega }\bar{\partial }_{E} (h^{-1}\partial_{K} h)+\sqrt{-1}\Lambda_{\omega } F_{K}-\lambda \cdot \textmd{Id}_E+\varepsilon \log h=0.
\end{equation}

\medskip

\begin{lem}[\cite{LY}]\label{lm1}
There exists a solution $H_{\varepsilon }$ to the perturbed equation (\ref{eq}) for all $\varepsilon >0$. And there hold that
\begin{enumerate}
  \item $-\frac{\sqrt{-1}}{2}\Lambda_{\omega }\partial \bar{\partial }\left(| \log{h_{\varepsilon }}|_{K}^2\right)+\vps | \log{h_{\varepsilon }} |_{K}^2\leq  | \sqrt{-1}\Lambda_{\omega } F_{K}-\lambda \cdot \Id_E |_{K} | \log{h_{\varepsilon }} |_{K};$
  \item  $\max_M | \log{h_{\varepsilon }} |_{K}\leq \frac{1}{\vps}\cdot \max_M  | \sqrt{-1}\Lambda_{\omega } F_{K}-\lambda \cdot \Id_E |_{K} $;
  \item $\max_M | \log{h_{\varepsilon }} |_{K}\leq C\cdot (\| \log{h_{\varepsilon }} \|_{L^2}+\max_M  | \sqrt{-1}\Lambda_{\omega } F_{K}-\lambda \cdot \Id_E |_{K})$,
\end{enumerate}
where $h_{\varepsilon }=K^{-1}H_{\varepsilon}$, $C$ is a constant depending only on $(M, \omega )$.
Moreover,  from $\tr (\sqrt{-1}\Lambda_{\omega } F_{K}-\lambda \cdot \Id_E)=0$, it holds that
\begin{equation}
\tr \log (h_{\varepsilon })=0
\end{equation}  and $\tr F_{H_{\varepsilon }} =\tr F_{K}$.
\end{lem}

\begin{prop}[{\cite[Proposition 3.1]{NZ}}]\label{key}
Let $(E, \bar{\partial}_{E})$ be a holomorphic vector
bundle with a fixed Hermitian metric $K $ over a compact Gauduchon manifold $(M, \omega )$ of complex dimension $n$. Assume $H$ is a Hermitian metric on $E$ and $s:=\log(K^{-1}H)$. Then we have
\begin{equation}\label{eq04021}
\int_M \mathrm{tr}((\sqrt{-1}\Lambda_{\omega } F_{K})s)\frac{\omega^n}{n!}+\int_{M}\langle \bar{\Psi}(s)(\bar{\partial}_{E}s),\bar{\partial}_{E}s\rangle_{K}\frac{\omega^n}{n!}=\int_M \mathrm{tr}((\sqrt{-1}\Lambda_{\omega } F_{H})s)\frac{\omega^n}{n!},
\end{equation}
where  $\bar{\Psi }$ is the function which is defined in (\ref{eq3301}).
\end{prop}

The above identity (\ref{eq04021}) also works for compact manifolds with nonempty boundary case and some noncompact manifolds case (see \cite[Proposition 2.6]{ZZZ}).

\medskip

Suppose $H_{\varepsilon}$ is the solution of the perturbed equation (\ref{eq}),  i.e.
\begin{equation}\label{eq1}
\sqrt{-1}\Lambda_{\omega}F_{H_{\varepsilon}}-\lambda\Id_E+ \varepsilon\log(K^{-1}H_{\varepsilon})=0.
\end{equation}
By Lemma \ref{lm1}, we have
\begin{equation}\label{P1}
\|\varepsilon \log(K^{-1}H_{\varepsilon})\|_{L^\infty} \leq \|\sqrt{-1}\Lambda_{\omega}F_{K}-\lambda\Id_E\|_{L^{\infty}},\end{equation}
\begin{equation}\label{P2}\tr\log(K^{-1}H_{\varepsilon})=0
\end{equation}
and
\begin{equation}\label{P12}\| \log(K^{-1}H_{\varepsilon})\|_{L^\infty} \leq C(\| \log(K^{-1}H_{\varepsilon})\|_{L^2}+\|\sqrt{-1}\Lambda_{\omega}F_{K}-\lambda\Id_E\|_{L^{\infty}}),\end{equation}
where $C$ is a constant depending only on the geometry of $(M, \omega )$.

According to \cite{NZ}, we know that $(E, \bar{\partial}_E)$ is $\omega$-semistable if and only if $\varliminf\limits_{\varepsilon \rightarrow  0}\| \varepsilon\log(K^{-1}H_{\varepsilon})\|_{L^2}=0$. If $(E, \bar{\partial}_E)$ is not $\omega$-semistable, we must have
\begin{equation}
\varliminf\limits_{\varepsilon \rightarrow  0}\| \varepsilon\log(K^{-1}H_{\varepsilon})\|_{L^2}=\tilde{\delta}> 0.
\end{equation}
Let's choose a sequence $\varepsilon_i\to 0$, as $i\rightarrow \infty$, such that
 \begin{equation}
 \lim\limits_{i\rightarrow \infty}\| \log(K^{-1}H_{\varepsilon_i})\|_{L^2}=+\infty \quad \text{and}\quad
 \lim\limits_{i\rightarrow \infty}\| \varepsilon_i\log(K^{-1}H_{\varepsilon_i})\|_{L^2}=\delta> 0.
 \end{equation}
 At present we don't require that $\delta =\tilde{\delta}$. Later, by (\ref{eq1}) and (\ref{370}), we will see that
 \begin{equation*}
 \begin{split}
 \delta^{2}&=
 \left\Vert\frac{2\pi}{\Vol(M,\omega)}\Phi^{HN}_\omega(E,K)-\lambda \Id_{E}\right\Vert_{L^2(K)}^{2}\\
 &=\sum_{j=1}^{r}\left|\frac{2\pi}{\Vol(M,\omega)}\mu_{j, \omega}(E)-\lambda \right|^{2}\Vol(M,\omega).\\
\end{split}
\end{equation*}
Thus $\delta $ is unique and consequently equal to $\tilde{\delta}$.
 In the sequel, we denote $H_{\varepsilon_i}$ by $H_{i}$ and set $h_i=K^{-1}H_{i}$, $s_i= \log h_i$, $l_i= \| s_i\|_{L^2}$, $u_i= \frac{s_i}{l_i}$ for simplicity. Then
  \begin{equation}
 \frac{\sqrt{-1}\Lambda_{\omega}F_{H_{i}}-\lambda\Id_E}{\varepsilon_i l_i}=-\frac{\varepsilon_i s_i}{\varepsilon_i l_i}=-u_i,
 \end{equation}
 $\tr u_i=0$ and $\|u_i\|_{L^2}=1$.
From (\ref{eq33}), one can see
 \begin{equation}\label{key0}
 \int_M (\tr((\sqrt{-1}\Lambda_{\omega}F_{K}-\lambda\Id_E)u_i)+l_i \langle\bar{\Psi}(l_i u_i)(\bar{\partial}u_i), \bar{\partial}u_i\rangle_K)\frac{\omega^n}{n!}=-\varepsilon_i l_i.
 \end{equation}

By (\ref{key0}) and following Simpson's argument (\cite[Lemma 5.4]{SIM}), we have
\begin{equation}
\|u_i\|_{L^\infty }\leq \hat{C} \quad \text{and} \quad \|D_{K} u_i\|_{L^2}< \tilde{C},
\end{equation}
i.e. $u_i$ are uniformly bounded in $L^\infty$ and $L_1^2$. So one can choose a subsequence, which is also denoted by $\{u_i\}$ for simplicity, such that $u_i \rightharpoonup u_\infty$ weakly in $L_1^2$. By Kondrachov compactness theorem (\cite[Theorem 7.22]{GT}), we know that $L_1^2$ is compactly embedded in $L^q$ for any $0<q< \frac{2n}{n-1}$. This tells us that
\begin{equation}
\lim_{i\rightarrow \infty} \|u_i-u_\infty\|_{L^q}=0
\end{equation}
and
\begin{equation}\label{lqcvg}
\lim_{i\rightarrow \infty} \|\sqrt{-1}\Lambda_{\omega}F_{H_{i}}-\lambda\Id_E+\delta u_\infty\|_{L^q}=0
\end{equation}
for any $0<q< \frac{2n}{n-1}$. Hence $\|u_\infty\|_{L^2}=1$.

Let $\mathcal{F}\subset E$ be a torsion-free subsheaf.
Note that $(\pi^{H_i}_{\mathcal{F}})^{\ast_{K}}= h_i (\pi^{H_i}_{\mathcal{F}})^{\ast_{H_i}} h_i^{-1}= h_i \pi^{H_i}_{\mathcal{F}} h_i^{-1}$ and $(h_i^{\frac{1}{2}}\pi^{H_i}_{\mathcal{F}}h_i^{-\frac{1}{2}})^{\ast_{K}}= h_i^{-\frac{1}{2}}(\pi^{H_i}_{\mathcal{F}})^{\ast_{K}}h_i^{\frac{1}{2}}= h_i^{\frac{1}{2}}\pi^{H_i}_{\mathcal{F}}h_i^{-\frac{1}{2}}$. Thus $|h_i^{\frac{1}{2}}\pi^{H_i}_{\mathcal{F}}h_i^{-\frac{1}{2}}|^2_K=\rank(\mathcal{F})$. Then
\begin{equation}
\begin{split}
2\pi\deg(\mathcal{F})=&\int_M(\tr(\pi^{H_i}_{\mathcal{F}}\sqrt{-1}\Lambda_{\omega}F_{H_{i}})-|\bar{\partial}\pi^{H_i}_{\mathcal{F}}|_{H_{i}}^2)\frac{\omega^n}{n!}\\
\leq&\int_M\tr(\pi^{H_i}_{\mathcal{F}}\sqrt{-1}\Lambda_{\omega}F_{H_{i}})\frac{\omega^n}{n!}\\
=&\int_M\tr(h_i^{\frac{1}{2}}\pi^{H_i}_{\mathcal{F}}h_i^{-\frac{1}{2}}h_i^{\frac{1}{2}}(\sqrt{-1}\Lambda_{\omega}F_{H_{i}})h_i^{-\frac{1}{2}})\frac{\omega^n}{n!}\\
=&\int_M\tr(h_i^{\frac{1}{2}}\pi^{H_i}_{\mathcal{F}}h_i^{-\frac{1}{2}}(\lambda\Id_E-\varepsilon_i\log h_i))\frac{\omega^n}{n!}\\
=&\lambda\cdot\rank(\mathcal{F})\cdot\Vol(M, \omega)+\int_M\tr(h_i^{\frac{1}{2}}\pi^{H_i}_{\mathcal{F}}h_i^{-\frac{1}{2}}\varepsilon_i l_i(u_\infty-u_i))\frac{\omega^n}{n!}\\
&-\int_M\tr(h_i^{\frac{1}{2}}\pi^{H_i}_{\mathcal{F}}h_i^{-\frac{1}{2}}\varepsilon_i l_i u_\infty)\frac{\omega^n}{n!},
\end{split}
\end{equation}
where we have used that $h_i^{\frac{1}{2}}(\sqrt{-1}\Lambda_{\omega}F_{H_{i}})h_i^{-\frac{1}{2}}=\sqrt{-1}\Lambda_{\omega}F_{H_{i}}$ under the condition that $\sqrt{-1}\Lambda_{\omega}F_{H_{i}}=\lambda\Id_E-\varepsilon_i\log h_i$.
Clearly there holds that when $i\rightarrow \infty$,
\begin{equation}
\begin{split}
&\int_M\tr(h_i^{\frac{1}{2}}\pi^{H_i}_{\mathcal{F}}h_i^{-\frac{1}{2}}\varepsilon_i l_i(u_\infty-u_i))\frac{\omega^n}{n!}\\
\leq &  \varepsilon_i l_i\cdot(\rank(\mathcal{F}))^{\frac{1}{2}}\int_M|u_\infty-u_i|_K\frac{\omega^n}{n!}\rightarrow 0.
\end{split}
\end{equation}
Again by (\ref{key0}) and following Simpson's argument (\cite[Lemma 5.5]{SIM}), one can check that the eigenvalues of $u_{\infty}$ are constants and not all equal.
Assume $\mu_1< \mu_2< \cdots< \mu_l$ are the distinct eigenvalues of $u_\infty$. Let $\{e_1, . . . , e_r\}$ be an orthonormal basis of $E$ with respect to $H_i$ at the considered point such that
\begin{equation}
\pi^{H_i}_{\mathcal{F}} e_{\alpha}=\left\{
\begin{split}
\ &e_{\alpha}, \quad &\alpha\leq \rank(\mathcal{F}),\\
\ &0, \quad &\alpha> \rank(\mathcal{F}).
\end{split}\right.
\end{equation}
Then $\langle h_i^{\frac{1}{2}}e_\alpha, h_i^{\frac{1}{2}}e_\beta\rangle_K= \langle h_i e_\alpha, e_\beta\rangle_K=\delta_{\alpha\beta}$. Set $\tilde{e}_\alpha= h_i^{\frac{1}{2}}e_\alpha$. Obviously $\{\tilde{e}_1, . . . , \tilde{e}_r\}$ is an orthonormal basis of $E$ with respect to $K$. It is easy to find that
\begin{equation}
\begin{split}
-\tr(h_i^{\frac{1}{2}}\pi^{H_i}_{\mathcal{F}}h_i^{-\frac{1}{2}}u_\infty)=&\sum_{\alpha=1}^{r}-\langle h_i^{\frac{1}{2}}\pi^{H_i}_{\mathcal{F}}h_i^{-\frac{1}{2}}u_\infty (\tilde{e}_\alpha), \tilde{e}_\alpha\rangle_K\\
=&\sum_{\alpha=1}^{\rank(\mathcal{F})}\langle -u_\infty \tilde{e}_\alpha, \tilde{e}_\alpha\rangle_K\leq -\mu_{1}\rank(\mathcal{F}).
\end{split}
\end{equation}
Thus
\begin{equation}
-\int_M\tr(h_i^{\frac{1}{2}}\pi^{H_i}_{\mathcal{F}}h_i^{-\frac{1}{2}}\varepsilon_i l_i u_\infty)\frac{\omega^n}{n!}\leq -\varepsilon_i l_i \mu_1 \cdot\rank(\mathcal{F})\cdot\Vol(M),
\end{equation}
and then
\begin{equation}\label{lambda1}
\frac{2\pi \deg(\mathcal{F})}{\rank(\mathcal{F})}\leq (\lambda-\delta\mu_1)\Vol(M).
\end{equation}

 For $A< l$, define a smooth function $P_A: \mathbb{R}\to \mathbb{R}$ such that
\begin{equation}
P_A(x)=\left\{
\begin{split}
\ 1, \quad & x\leq \mu_A,\\
\ 0, \quad & x\geq \mu_{A+1}.
\end{split}\right.
\end{equation}
Setting $\pi_A=P_A(u_\infty)$, by the argument as  in \cite[p.~887]{SIM}, we have
\begin{enumerate}
  \item $\pi_{A}\in L^2_1;$
  \item $\pi^2_{A}=\pi_{A}=\pi_{A}^{*_K};$
  \item $(\mathrm{Id}_E-\pi_{A})\bp \pi_{A}=0.$
  \end{enumerate}
According to the regularity statement for $L_1^2$-subbundles in \cite{UY86}, we know that $\pi_A$ defines a saturated subsheaf $E_A$ of $E$ (i.e. subsheaf with torsion-free quotient). Away from the singular set $Sing(E_A)$, $E_A$ is a holomorphic subbundle of $E$.  We also set $E_0={0}$ and $E_l=E$. In the following, write $r_A= \rank(E_A)$ for simplicity.

\begin{lem}\label{L1}
We have
\begin{equation}
-\tr(h_i^{\frac{1}{2}}\pi^{H_i}_{E_A}h_i^{-\frac{1}{2}}u_\infty)\leq \sum_{B=1}^{A}(-\mu_B)(\rank(E_B)-\rank(E_{B-1})),
\end{equation}
where $\pi^{H_i}_{E_A}$ is the orthogonal projection onto $E_A$ with respect to $H_i$, then
\begin{equation}\label{subdeg}
2\pi \deg(E_A)\leq \Vol(M) \sum_{B=1}^{A}(\lambda-\delta\mu_B)(\rank(E_B)-\rank(E_{B-1})).
\end{equation}
\end{lem}

\begin{proof} At $x\in M\setminus Sing(E_A)$,  there is a basis $\{e_1, \cdots, e_{r_A}\}$ of  $E_A|_x$. We choose $\{\check{e}_1, \cdots, \check{e}_{r_A}\}$ as an orthonormal basis of $E_A|_x$ with respect to $H_i$, and extend it to $\{\check{e}_1, \cdots, \check{e}_{r_A}, \cdots, \check{e}_r\}$ as an orthonormal basis of $E|_x$ with respect to $H_i$. Set $\hat{e}_\alpha= h_i^{\frac{1}{2}}\check{e}_\alpha$, so $\langle \hat{e}_\alpha, \hat{e}_\beta\rangle_K= \delta_{\alpha\beta}$, i.e. $\{\hat{e}_\alpha\}_{\alpha=1}^{r}$ is an orthonormal basis with respect to $K$. Define $\pi_l= \Id_E$ and $\pi_0=0$. Then one has the fact that \begin{equation}\label{infty}u_\infty= \sum_{B=1}^{l}\mu_B(\pi_B-\pi_{B-1}),\end{equation} where $\pi_B$ defined as above is the orthogonal projection onto $E_B$ with respect to $K$.

A straightforward calculation yields that
\begin{equation}
\begin{split}
&-\tr(h_i^{\frac{1}{2}}\pi^{H_i}_{E_A}h_i^{-\frac{1}{2}}u_\infty)\\
=&-\sum_{\alpha=1}^r\langle h_i^{\frac{1}{2}}\pi^{H_i}_{E_A}h_i^{-\frac{1}{2}}u_\infty(\hat{e}_\alpha), \hat{e}_\alpha\rangle_K \\
=&\sum_{\alpha=1}^{r_A}\langle -u_\infty(\hat{e}_\alpha),  \hat{e}_\alpha\rangle_K \\
=&\sum_{\alpha=1}^{r_A}\sum_{B=1}^{l}\langle (\mu_A-\mu_B)(\pi_B-\pi_{B-1})(\hat{e}_\alpha),  \hat{e}_\alpha\rangle_K- \mu_A\cdot r_A\\
\leq &\sum_{\alpha=1}^{r_A}\sum_{B=1}^{A-1}\langle (\mu_A-\mu_B)(\pi_B-\pi_{B-1})(\hat{e}_\alpha),  \hat{e}_\alpha\rangle_K- \mu_A\cdot r_A\\
\leq &\sum_{B=1}^{A-1}(\mu_A-\mu_B)\sum_{\alpha=1}^{r}\langle (\pi_B-\pi_{B-1})(\hat{e}_\alpha),  \hat{e}_\alpha\rangle_K- \mu_A\cdot r_A\\
= & \sum_{B=1}^{A}(-\mu_B)(\rank(E_B)-\rank(E_{B-1})),
\end{split}
\end{equation}
where the first inequality comes from the facts that $\mu_A-\mu_B\leq 0$ if $B\geq A$ and $\langle (\pi_B-\pi_{B-1})(\hat{e}_\alpha),  \hat{e}_\alpha\rangle_K\geq 0$, which is due to $(\pi_B-\pi_{B-1})^2=\pi_B^2-\pi_B\circ\pi_{B-1}-\pi_{B-1}\circ\pi_B+ \pi_{B-1}^2= \pi_B-\pi_{B-1}$ and $(\pi_B-\pi_{B-1})^{\ast_{K}}= \pi_B-\pi_{B-1}$, in the last equality  we have used $\tr\pi_B=\rank(E_B)$. Then
\begin{equation}
\begin{split}
2\pi\deg(E_A)=& \int_M (\tr(\pi_{E_A}^{H_i}\sqrt{-1}\Lambda_{\omega}F_{H_i})-|\bar{\partial}\pi_{E_A}^{H_i}|_{H_{i}}^2)\frac{\omega^n}{n!}\\
\leq& \int_M \tr(h_i^{\frac{1}{2}}\pi_{E_A}^{H_i}h_i^{-\frac{1}{2}}h_i^{\frac{1}{2}}\sqrt{-1}\Lambda_{\omega}F_{H_i}h_i^{-\frac{1}{2}})\frac{\omega^n}{n!}\\
=& \int_M \tr(h_i^{\frac{1}{2}}\pi_{E_A}^{H_i}h_i^{-\frac{1}{2}}(\lambda\Id_E-\varepsilon_i l_iu_i))\frac{\omega^n}{n!}\\
=& \lambda\cdot\rank(E_A)\cdot\Vol(M, \omega)\\
&+ \int_M \tr(h_i^{\frac{1}{2}}\pi_{E_A}^{H_i}h_i^{-\frac{1}{2}}\varepsilon_i l_i(u_\infty-u_i))\frac{\omega^n}{n!}\\
&-\int_M \tr(h_i^{\frac{1}{2}}\pi_{E_A}^{H_i}h_i^{-\frac{1}{2}}\varepsilon_i l_iu_\infty)\frac{\omega^n}{n!}\\
\leq& \lambda\cdot\rank(E_A)\cdot\Vol(M, \omega)\\
&+ \int_M \tr(h_i^{\frac{1}{2}}\pi_{E_A}^{H_i}h_i^{-\frac{1}{2}}\varepsilon_i l_i(u_\infty-u_i))\frac{\omega^n}{n!}\\
&-\varepsilon_i l_i \Vol(M, \omega)(\sum_{B=1}^A\mu_B(\rank(E_B)-\rank(E_{B-1}))).
\end{split}
\end{equation}
Therefore, we achieve (\ref{subdeg}).
\end{proof}

For simplicity, we write $\lambda_A= \lambda-\delta\mu_A$. Then it follows that $\lambda_1> \lambda_2> \cdots> \lambda_l$. For any torsion-free subsheaf $\mathcal{F}\subset E$, by (\ref{lambda1}), we know that
\begin{equation}
\frac{2\pi \deg(\mathcal{F})}{\rank(\mathcal{F})}\leq \lambda_1\Vol(M, \omega).
\end{equation}
Now consider the exact sequence
\begin{equation}
0 \longrightarrow \mathcal{F} \longrightarrow E \longrightarrow \mathcal{Q} \longrightarrow 0.
\end{equation}
There holds that
\begin{equation}
\begin{split}
2\pi\deg(\mathcal{Q})=& \int_M(\tr((\Id_E-\pi^{H_i}_{\mathcal{F}})\sqrt{-1}\Lambda_{\omega}F_{H_{i}})+|\bar{\partial}\pi^{H_i}_{\mathcal{F}}|_{H_{i}}^2)\frac{\omega^n}{n!}\\
\geq &\int_M\tr(h_i^{\frac{1}{2}}(\Id_E-\pi^{H_i}_{\mathcal{F}})h_i^{-\frac{1}{2}}\cdot\sqrt{-1}\Lambda_{\omega}F_{H_{i}})\frac{\omega^n}{n!}\\
=&\int_M\tr(h_i^{\frac{1}{2}}(\Id_E-\pi^{H_i}_{\mathcal{F}})h_i^{-\frac{1}{2}}(\sqrt{-1}\Lambda_{\omega}F_{H_{i}}-(\lambda\Id_E-\varepsilon_i l_i u_\infty)))\frac{\omega^n}{n!}\\
&+\int_M\tr(h_i^{\frac{1}{2}}(\Id_E-\pi^{H_i}_{\mathcal{F}})h_i^{-\frac{1}{2}}(\lambda\Id_E-\varepsilon_i l_i u_\infty))\frac{\omega^n}{n!}.
\end{split}
\end{equation}
Take $i\to \infty$, then
\begin{equation}
2\pi\deg(\mathcal{Q})\geq\lambda_l\cdot\rank(\mathcal{Q})\cdot\Vol(M, \omega).
\end{equation}
Apply the same argument to the exact sequence
\begin{equation}
0 \longrightarrow E_B \longrightarrow E \longrightarrow E/E_B \longrightarrow 0.
\end{equation}
Then
\begin{equation}\label{deg}
\begin{split}
&2\pi\deg(E/E_B)\\
=&2\pi(\deg(E)-\deg(E_B))\\
=&\int_M(\tr((\Id_E-\pi^{H_i}_{E_B})\sqrt{-1}\Lambda_{\omega}F_{H_{i}})+|\bar{\partial}\pi^{H_i}_{E_B}|_{H_{i}}^2)\frac{\omega^n}{n!}\\
\geq&\int_M\tr(h_i^{\frac{1}{2}}(\Id_E-\pi^{H_i}_{E_B})h_i^{-\frac{1}{2}}(\sqrt{-1}\Lambda_{\omega}F_{H_{i}}-(\lambda\Id_E-\delta u_\infty)))\frac{\omega^n}{n!}\\
&+\int_M\tr(h_i^{\frac{1}{2}}(\Id_E-\pi^{H_i}_{E_B})h_i^{-\frac{1}{2}}(\lambda\Id_E-\delta u_\infty))\frac{\omega^n}{n!}.
\end{split}
\end{equation}

After a similar computation as in Lemma \ref{L1}, one can see

\begin{lem}\label{L2}
\begin{equation}\label{quodeg}
2\pi\deg(E/E_B)\geq \sum_{A=B+1}^l \lambda_A(\rank(E_A)- \rank(E_{A-1}))\Vol(M, \omega).
\end{equation}
\end{lem}
\begin{proof} Note that $\lambda_1> \lambda_2> \cdots> \lambda_l$ and
\begin{equation}
0\subset E_1 \subset E_2 \subset\cdots \subset E_l= E.
\end{equation}
At the point on the locally free part, let $\{\check{e}_1, \cdots, \check{e}_{r_B}\}$ be an orthonormal basis of $E_B$ with respect to $H_i$, and extend it to $\{\check{e}_1, \cdots, \check{e}_{r_B}, \cdots, \check{e}_r\}$ as the orthonormal basis of $E$ with respect to $H_i$. Set $\hat{e}_\alpha= h_i^{\frac{1}{2}}\check{e}_\alpha$, then $\langle \hat{e}_\alpha, \hat{e}_\beta\rangle_K= \delta_{\alpha\beta}$, i.e. $\{\hat{e}_\alpha\}_{\alpha=1}^{r}$ is an orthonormal basis with respect to $K$.

Recall $u_\infty= \sum_{A=1}^{l}\mu_A(\pi_A-\pi_{A-1})$, where $\pi_A$ is the orthogonal projection onto $E_A$ with respect to $K$, $\pi_0=0$ and $\pi_l= \Id_E$. Denote $\tilde{u}_\infty=\lambda\Id_E-\delta u_\infty$ and then \begin{equation}\label{infty2}\tilde{u}_\infty= \sum_{A=1}^{l}\lambda_A(\pi_A-\pi_{A-1}).
\end{equation}
Directly calculating gives that
\begin{equation}
\begin{split}
&\tr(h_i^{\frac{1}{2}}(\Id_E-\pi^{H_i}_{E_B})h_i^{-\frac{1}{2}}\tilde{u}_\infty)\\
=& \sum_{\alpha=r_B+1}^r \langle \tilde{u}_\infty\hat{e}_\alpha, \hat{e}_\alpha\rangle_K\\
=& \sum_{\alpha=r_B+1}^r \langle (\tilde{u}_\infty-\lambda_{B} \Id_E)\hat{e}_\alpha, \hat{e}_\alpha\rangle+\lambda_{B}(r- r_B)\\
=& \sum_{\alpha=r_B+1}^r \sum_{A=1}^l\langle (\lambda_A-\lambda_{B})(\pi_A-\pi_{A-1})\hat{e}_\alpha, \hat{e}_\alpha\rangle+\lambda_{B}(r- r_B)\\
\geq& \sum_{\alpha=r_B+1}^r \sum_{A=B+1}^l\langle (\lambda_A-\lambda_{B})(\pi_A-\pi_{A-1})\hat{e}_\alpha, \hat{e}_\alpha\rangle+\lambda_{B}(r- r_B)\\
\geq& \sum_{A=B+1}^l (\lambda_A-\lambda_{B})\tr(\pi_A-\pi_{A-1})+\lambda_{B}(r- r_B)\\
=& \sum_{A=B+1}^l \lambda_A(r_A-r_{A-1}).
\end{split}
\end{equation}
So
\begin{equation}
\begin{split}
&\int_M\tr(h_i^{\frac{1}{2}}(\Id_E-\pi^{H_i}_{E_B})h_i^{-\frac{1}{2}}(\lambda\Id_E-\delta u_\infty))\frac{\omega^n}{n!}\\
\geq & \Vol(M)\sum_{A=B+1}^l (\lambda-\delta \mu_A)(r_A-r_{A-1}).
\end{split}
\end{equation}
Putting this into (\ref{deg}) and letting $i\to \infty$, we get the desired inequality (\ref{quodeg}).
\end{proof}

Recall that $\lambda_{mU}(H_i,\omega)$  is the average of the largest eigenvalue function $\lambda_U(H_i,\omega)$ of $\sqrt{-1}\Lambda_{\omega}F_{H_{i}}$, $\lambda_{mL}(H_i,\omega)$ is the average of the smallest eigenvalue function $\lambda_L(H_i,\omega)$ of $\sqrt{-1}\Lambda_{\omega}F_{H_{i}}$. By the Chern-Weil formula (\ref{B1}), it is easy to verify that
\begin{equation}\label{lambda2}
\varliminf_{i\to \infty}\lambda_{mU}(H_i,\omega)\Vol(M,\omega)\geq \sup_{\mathcal{F}\subset E} 2\pi(\frac{\deg(\mathcal{F})}{\rank(\mathcal{F})}),
\end{equation}
where $\mathcal{F}$ runs over all the subsheaves of $E$, and
\begin{equation}\label{lambda3}
\varlimsup_{i\to \infty}\lambda_{mL}(H_i,\omega)\Vol(M,\omega)\leq \inf_{\mathcal{Q}} 2\pi(\frac{\deg(\mathcal{Q})}{\rank(\mathcal{Q})}),
\end{equation}
where $\mathcal{Q}$ runs over all the quotient sheaves of $E$.

Furthermore, we have:

\begin{lem}

\begin{equation}\label{lambda4}
\varlimsup_{i\to \infty}\lambda_{mU}(H_i,\omega)\leq \lambda_1
\end{equation}
and
\begin{equation}\label{lambda5}
\varliminf_{i\to \infty}\lambda_{mL}(H_i,\omega)\geq \lambda_l.
\end{equation}
\end{lem}
\begin{proof}
Suppose $e_1^{i}$ is an eigenvector of $\sqrt{-1}\Lambda_{\omega}F_{H_{i}}$ with respect to $\lambda_U(H_i,\omega)$, and $|e_1^{i}|_K=1$. Of course one has
\begin{equation}
\begin{split}
\lambda_U(H_i,\omega)=&\langle \sqrt{-1}\Lambda_{\omega}F_{H_{i}}(e_1^{i}), e_1^{i}\rangle_K\\
=& \langle (\sqrt{-1}\Lambda_{\omega}F_{H_{i}}-(\lambda\Id_E- \delta u_\infty))e_1^{i}, e_1^{i}\rangle_K
+\langle (\lambda\Id_E- \delta u_\infty)e_1^{i}, e_1^{i}\rangle_K\\
\leq& |\sqrt{-1}\Lambda_{\omega}F_{H_{i}}-(\lambda\Id_E- \delta u_\infty)|_{K}+ \lambda_1.
\end{split}
\end{equation}
This means that
\begin{equation}
\varlimsup_{i\to \infty}\lambda_{mU}(H_i,\omega)\Vol(M,\omega)
=\varlimsup_{i\to \infty}\int_M \lambda_U(H_i,\omega)\frac{\omega^n}{n!}
\leq \lambda_1 \Vol(M, \omega).
\end{equation}
Immediately (\ref{lambda5}) can be proved in a similar way.
\end{proof}

Define
\begin{equation}
\nu= 2\pi\sum_{A=1}^{l-1}(\mu_{A+1}-\mu_{A})\rank(E_A)(\frac{\deg(E)}{\rank(E)}-\frac{\deg(E_A)}{\rank(E_A)}),
\end{equation}
then
\begin{equation}
\begin{split}
\nu=& 2\pi(\mu_l\deg(E)-\sum_{A=1}^{l-1}(\mu_{A+1}-\mu_{A})\deg(E_A))\\
=& 2\pi(\mu_l\deg(E)+\sum_{A=1}^{l-1}\mu_{A}\deg(E_A)-\sum_{A=2}^{l}\mu_{A}\deg(E_{A-1}))\\
=& 2\pi\sum_{A=1}^{l}\mu_{A}(\deg(E_A)-\deg(E_{A-1})).
\end{split}
\end{equation}
The fact $\|u_\infty\|_{L^2}=1$ yields that
\begin{equation}
\sum_{A=1}^l\mu_A^2(\rank(E_A)-\rank(E_{A-1}))\Vol(M)=1.
\end{equation}
Recall $\mu_A=\frac{\lambda-\lambda_A}{\delta}$. Evidently it holds that
\begin{equation}
\sum_{A=1}^l(\lambda-\lambda_A)^2(\rank(E_A)-\rank(E_{A-1}))=\frac{\delta^2}{\Vol(M)}.
\end{equation}

By (\ref{key0}) and the same discussion in \cite[Lemma 5.4]{SIM} (\cite[(3.23)]{NZ}), we know
\begin{equation}\label{key00}
\delta+ \int_M(\tr(u_\infty \sqrt{-1}\Lambda_{\omega}F_K)+\langle \zeta(u_\infty)\bar{\partial}u_\infty, \bar{\partial}u_\infty\rangle_K)\frac{\omega^n}{n!}\leq 0,
\end{equation}
where $ \zeta \in C^{\infty} (\mathbb R\times \mathbb R, \mathbb R^+)$ satisfies $\zeta(x,y)<(x-y)^{-1}$ whenever $x>y$.
Notice that
\begin{equation}
2\pi\deg(E_B)=\int_M(\tr(\pi_B\cdot \sqrt{-1}\Lambda_{\omega}F_K)- |\bar{\partial}\pi_B|_K^2)\frac{\omega^n}{n!}.
\end{equation}
 So by (\ref{key00}) and following the arguments in \cite[p.~793-794]{LZ}, we obtain
\begin{equation}\label{eq320}
\begin{split}
\nu = &\int_M \tr(u_{\infty}\sqrt{-1}\Lambda_{\omega}F_K)+ \lag \sum^{l-1}_{A=1}(\mu_{A+1}-\mu_{A})
(dP_{A})^2(u_{\infty})(\bar{\partial }u_{\infty}), \bar{\partial }u_{\infty}\rag_{K}\frac{\omega^{n}}{n!}\\
\leq & -\delta,
\end{split}
\end{equation}
where the function $dP_{A}: \mathbb R\times \mathbb R\lraw \mathbb{R}$ is defined by
\begin{equation*}
dP_{A}(x,y)=\left\{
\begin{split}
&\frac{P_{A}(x)-P_{A}(y)}{x-y},\qquad  x\neq y;\\
&\qquad P_{A}'(x), \qquad \qquad x=y.
\end{split}\right.
\end{equation*}
Taking into account $\tr u_\infty \equiv 0$, one has
\begin{equation} \sum_{A=1}^l \mu_A(\rank(E_A)-\rank(E_{A-1}))=0.\end{equation}Then
\begin{equation}\label{key1}
\begin{split}
0\geq &\delta^2+\delta\nu\\
=&\delta^2+2\pi\sum_{A=1}^l(\lambda-\lambda_A)(\deg(E_A)-\deg(E_{A-1}))\\
=&\sum_{A=1}^l(\lambda-\lambda_A)(2\pi(\deg(E_A)-\deg(E_{A-1}))-\lambda_A(r_A-r_{A-1})\Vol(M)).
\end{split}
\end{equation}
At the same time, we can conclude that
\begin{lem}
\begin{equation}\label{346}
\sum_{A=1}^l(\lambda-\lambda_A)(2\pi(\deg(E_A)-\deg(E_{A-1}))-\lambda_A(r_A-r_{A-1})\Vol(M))\geq 0.
\end{equation}
\end{lem}
\begin{proof}
Computing straightforwardly gives that
\begin{equation}\label{key2}
\begin{aligned}
&\sum_{A=1}^{l}(\lambda-\lambda_A)(2\pi(\deg(E_A)-\deg(E_{A-1}))-\Vol(M)\cdot\lambda_A(r_A-r_{A-1}))\\
=&\sum_{A=1}^{l}(\lambda-\lambda_A)\bigg(2\pi\deg(E_A)-\Vol(M)\cdot\sum_{B=1}^A\lambda_B(r_B-r_{B-1})\\
&-\bigg(2\pi\deg(E_{A-1})-\Vol(M)\cdot\sum_{B=1}^{A-1}\lambda_B(r_B-r_{B-1})\bigg)\bigg)\\
=&\sum_{A=1}^{l}(\lambda-\lambda_A)\bigg(2\pi\deg(E_A)-\Vol(M)\cdot\sum_{B=1}^A\lambda_B(r_B-r_{B-1})\bigg)\\
&-\sum_{A=1}^{l-1}(\lambda-\lambda_{A+1})\bigg(2\pi\deg(E_{A})-\Vol(M)\cdot\sum_{B=1}^A\lambda_B(r_B-r_{B-1})\bigg)\\
=&\sum_{A=1}^{l-1}(\lambda-\lambda_{A}-(\lambda-\lambda_{A+1}))\bigg(2\pi\deg(E_A)-\Vol(M)\cdot\sum_{B=1}^A\lambda_B(r_B-r_{B-1})\bigg)\\
&+(\lambda-\lambda_l)\bigg(2\pi\deg(E)-\Vol(M)\cdot\sum_{A=1}^l\lambda_B(r_B-r_{B-1})\bigg)\\
=&\sum_{A=1}^{l-1}(\lambda_{A+1}-\lambda_A)\bigg(2\pi\deg(E_A)-\Vol(M)\cdot\sum_{B=1}^A\lambda_B(r_B-r_{B-1})\bigg)\\
\geq &0,
\end{aligned}
\end{equation}
where the inequality is based on (\ref{subdeg}) and $\lambda_{A+1}<\lambda_{A}$, in the last equality we have used
\begin{equation}\label{degeq}
2\pi\deg(E)=\Vol(M)\cdot\sum_{A=1}^l\lambda_A(r_A-r_{A-1}).\qedhere
\end{equation}
\end{proof}

Since $\lambda_{A+1}<\lambda_A$, combining (\ref{subdeg}), (\ref{key1}), (\ref{key2}) and (\ref{degeq}), one can find that
\begin{equation}
2\pi\deg(E_A)= \Vol(M)\cdot\sum_{B=1}^A\lambda_B(r_B-r_{B-1})
\end{equation}
for $1\leq A\leq l$. Consequently we have
\begin{equation}\label{lambda06}
\frac{2\pi(\deg(E_A)-\deg(E_{A-1}))}{r_A-r_{A-1}}= \Vol(M)\cdot\lambda_A.
\end{equation}

By (\ref{lambda2}), (\ref{lambda3}), (\ref{lambda4}) and (\ref{lambda5}), we establish
\begin{equation}
\begin{split}
\sup_{\mathcal{F}\subset E} 2\pi(\frac{\deg(\mathcal{F})}{\rank(\mathcal{F})})\leq &\varliminf_{i\to \infty}\lambda_{mU}(H_i,\omega)\Vol(M,\omega)\\
\leq &\lambda_1 \Vol(M, \omega)= 2\pi\cdot\frac{\deg(E_1)}{\rank(E_1)}\\
\leq &\sup_{\mathcal{F}\subset E} 2\pi(\frac{\deg(\mathcal{F})}{\rank(\mathcal{F})})
\end{split}
\end{equation}
and
\begin{equation}
\begin{split}
\inf_{\mathcal{Q}} 2\pi(\frac{\deg(\mathcal{Q})}{\rank(\mathcal{Q})})\geq & \varlimsup_{i\to \infty}\lambda_{mL}(H_i,\omega)\Vol(M, \omega)\\
\geq &\lambda_l \Vol(M, \omega)= 2\pi\cdot\frac{\deg(E)-\deg(E_{l-1})}{\rank(E)-\rank(E_{l-1})}\\
\geq &\inf_{\mathcal{Q}} 2\pi(\frac{\deg(\mathcal{Q})}{\rank(\mathcal{Q})}).
\end{split}
\end{equation}
Hence it follows that
\begin{equation}\label{355}
\begin{split}
\lim_{i\to \infty}\lambda_{mU}(H_i,\omega)\Vol(M,\omega) & =\lambda_1 \Vol(M, \omega)= 2\pi\cdot\frac{\deg(E_1)}{\rank(E_1)}\\
& = \max_{\mathcal{F}\subset E} 2\pi(\frac{\deg(\mathcal{F})}{\rank(\mathcal{F})})
\end{split}
\end{equation}
and
\begin{equation}\label{356}
\begin{split}
\lim_{i\to \infty}\lambda_{mL}(H_i,\omega)\Vol(M,\omega) & = \lambda_l \Vol(M, \omega)= 2\pi\cdot\frac{\deg(E/E_{l-1})}{\rank(E/E_{l-1})}\\
& = \min_{\mathcal{Q}} 2\pi(\frac{\deg(\mathcal{Q})}{\rank(\mathcal{Q})}).
\end{split}
\end{equation}

Assume $\mathcal{F}$ is a subsheaf of $E$ with $\rank(\mathcal{F})> r_{A-1}$ for some $A\geq 2$. Clearly we have already known
\begin{equation}
\begin{split}
2\pi\deg(\mathcal{F})=&\int_M(\tr(\pi^{H_i}_{\mathcal{F}}\sqrt{-1}\Lambda_{\omega}F_{H_{i}})-|\bar{\partial}\pi^{H_i}_{\mathcal{F}}|_{H_{i}}^2)\frac{\omega^n}{n!}\\
\leq&\int_M\tr(h_i^{\frac{1}{2}}\pi^{H_i}_{\mathcal{F}}h_i^{-\frac{1}{2}}(\sqrt{-1}\Lambda_{\omega}F_{H_{i}}-\tilde{u}_\infty))\frac{\omega^n}{n!}\\
&+ \int_M\tr(h_i^{\frac{1}{2}}\pi^{H_i}_{\mathcal{F}}h_i^{-\frac{1}{2}}\tilde{u}_\infty)\frac{\omega^n}{n!}.
\end{split}
\end{equation}
Notice that $\mathcal{F}$ is a subbundle of $E$ away from the singular set $Sing(\mathcal{F})$. Suppose $x\in M\setminus Sing(\mathcal{F})$. We choose $\{\check{e}_1, \cdots, \check{e}_{\rank(\mathcal{F})}\}$ as the $H_i$-orthonormal basis of $\mathcal{F}|_x$, and extend it to $\{\check{e}_1, \cdots, \check{e}_{\rank(\mathcal{F})}, \cdots, \check{e}_r\}$ as the $H_i$-orthonormal basis of $E|_x$. Set $\hat{e}_\alpha= h_i^{\frac{1}{2}}\check{e}_\alpha$, so $\langle \hat{e}_\alpha, \hat{e}_\beta\rangle_K= \delta_{\alpha\beta}$, i.e. $\{\hat{e}_\alpha\}_{\alpha=1}^{r}$ is an orthonormal basis with respect to $K$.
As before, we also have
\begin{equation}
\begin{split}
&\tr(h_i^{\frac{1}{2}}\pi^{H_i}_{\mathcal{F}}h_i^{-\frac{1}{2}}\tilde{u}_\infty)\\
=&\sum_{\alpha=1}^{\rank(\mathcal{F})}\langle \tilde{u}_\infty(\hat{e}_\alpha),  \hat{e}_\alpha\rangle_K \\
=&\sum_{\alpha=1}^{\rank(\mathcal{F})}\langle (\tilde{u}_\infty-\lambda_A\Id_E)\hat{e}_\alpha,  \hat{e}_\alpha\rangle_K+ \lambda_A\cdot \rank(\mathcal{F})\\
=&\sum_{\alpha=1}^{\rank(\mathcal{F})}\sum_{B=1}^{l}(\lambda_B-\lambda_A)\langle (\pi_B-\pi_{B-1})(\hat{e}_\alpha),  \hat{e}_\alpha\rangle_K+ \lambda_A\cdot \rank(\mathcal{F})\\
\leq &\sum_{\alpha=1}^{\rank(\mathcal{F})}\sum_{B=1}^{A-1}(\lambda_B-\lambda_A)\langle (\pi_B-\pi_{B-1})(\hat{e}_\alpha),  \hat{e}_\alpha\rangle_K+ \lambda_A\cdot \rank(\mathcal{F})\\
\leq &\sum_{B=1}^{A-1}(\lambda_B-\lambda_A)\sum_{\alpha=1}^{r}\langle (\pi_B-\pi_{B-1})(\hat{e}_\alpha),  \hat{e}_\alpha\rangle_K+ \lambda_A\cdot \rank(\mathcal{F})\\
= & \sum_{B=1}^{A-1}(\lambda_B-\lambda_A)\tr(\pi_B-\pi_{B-1})+ \lambda_A\cdot \rank(\mathcal{F})\\
= & \sum_{B=1}^{A-1}\lambda_B(r_B-r_{B-1})+ \lambda_A\cdot (\rank(\mathcal{F})-r_{A-1}).
\end{split}
\end{equation}

Then
\begin{equation}\label{4.1.1}
\begin{split}
2\pi\deg(\mathcal{F})\leq &(\sum_{B=1}^{A-1}\lambda_B(r_B-r_{B-1})+ \lambda_A\cdot (\rank(\mathcal{F})-r_{A-1}))\Vol(M)\\
=& 2\pi\sum_{B=1}^{A-1}(\deg(E_B)-\deg(E_{B-1}))\\
& + \lambda_A\cdot (\rank(\mathcal{F})-r_{A-1})\Vol(M)\\
=&  2\pi \deg(E_{A-1})+ \lambda_A\cdot (\rank(\mathcal{F})-r_{A-1})\Vol(M).
\end{split}
\end{equation}
It follows that
\begin{equation}\label{HN11}
\begin{split}
\frac{2\pi(\deg(\mathcal{F})-\deg(E_{A-1}))}{\rank(\mathcal{F})-\rank(E_{A-1})}\leq & \frac{2\pi(\deg(E_A)-\deg(E_{A-1}))}{\rank(E_A)-\rank(E_{A-1})}\\
< & \lambda_{A-1}\Vol(M).
\end{split}
\end{equation}

 Next we are going to show that
\begin{equation*}
0= E_0\subset E_1 \subset E_2 \subset\cdots \subset E_l=E
\end{equation*}
is exactly the Harder-Narasimhan filtration of $(E, \bar{\partial}_E)$.
Obviously (\ref{355}) tells us that
\begin{equation}
\frac{\deg(E_1)}{\rank(E_1)}= \max_{\mathcal{F}\subset E} (\frac{\deg(\mathcal{F})}{\rank(\mathcal{F})}).
\end{equation}
If $\rank(\mathcal{F})> \rank(E_1)$,  from (\ref{HN11}), we get
\begin{equation}
\deg(\mathcal{F})-\deg (E_1)<\frac{\rank(\mathcal{F})-\rank(E_1)}{\rank(E_1)}\deg(E_1),
\end{equation} and then \begin{equation}\frac{\deg(\mathcal{F})}{\rank(\mathcal{F})}< \frac{\deg(E_1)}{\rank(E_1)}.\end{equation}
Consider
\begin{equation}
0 \subset E_B \subset \hat{\mathcal{F}} \subset E,
\end{equation}
where $\rank (\hat{\mathcal{F}})>\rank (E_{B})$ and $B\geq 1$.
Using (\ref{HN11}) again, one can see
\begin{equation}
\frac{\deg(\hat{\mathcal{F}})-\deg(E_{B})}{\rank(\hat{\mathcal{F}})-\rank(E_{B})}\leq \frac{\deg(E_{B+1})-\deg(E_{B})}{\rank(E_{B+1})-\rank(E_{B})},
\end{equation}
and if $\rank (\hat{\mathcal{F}})>\rank (E_{B+1})$, then
\begin{equation}
\frac{\deg(\hat{\mathcal{F}})-\deg(E_{B})}{\rank(\hat{\mathcal{F}})-\rank(E_{B})}< \frac{\deg(E_{B+1})-\deg(E_{B})}{\rank(E_{B+1})-\rank(E_{B})}.
\end{equation}
Therefore, we confirm that
\begin{equation*}
0= E_0\subset E_1 \subset E_2 \subset\cdots \subset E_l=E
\end{equation*}
is the Harder-Narasimhan filtration of $(E, \bar{\partial}_E)$.

\begin{proof}[Proof of Theorem \ref{theorem1}] By the previous argument, there holds
\begin{equation}
\begin{aligned}
\lambda \Id_E-\delta u_\infty&=\lambda\sum_{A=1}^l(\pi_A-\pi_{A-1})-\delta\sum_{A=1}^l\mu_A(\pi_A-\pi_{A-1})\\
&=\sum_{A=1}^l\lambda_A(\pi_A-\pi_{A-1})\\
&=\frac{2\pi}{\Vol(M,\omega)}\sum_{A=1}^l\mu_{\omega}(E_A/E_{A-1})(\pi_A-\pi_{A-1})\\
&=\frac{2\pi}{\Vol(M,\omega)}\Phi^{HN}_\omega(E, K).
\end{aligned}
\end{equation}
Together with (\ref{lqcvg}), we have for any $0<q<\frac{2n}{n-1}$,
\begin{equation}\label{370}
\lim_{i\rightarrow \infty} \left\Vert\sqrt{-1}\Lambda_{\omega}F_{H_{i}}-\frac{2\pi}{\Vol(M,\omega)}\Phi^{HN}_\omega(E,K)\right\Vert_{L^q(K)}=0.
\end{equation}
Since $|\sqrt{-1}\Lambda_{\omega} F_{H_{i}}|_K$ is uniformly bounded, (\ref{eqthm1}) follows.
\end{proof}

\section{HN-positivity and rational connectedness} Theorem \ref{RC} provides a metric criterion for rational connectedness.
\begin{lem}\label{rcslop}
 Let $M$ be a K\"ahler manifold. Then the following statements are equivalent:
\begin{enumerate}
\item[(1)] $M$ is projective and rationally connected;
\item[(2)] There exists a Gauduchon (resp. balanced)  metric $\omega$ on $M$ such that $\mu_L(T^{1,0}M,\omega)>0$, i.e. $T^{1,0}(M)$ is HN-positivity.
\end{enumerate}
\end{lem}
According to Theorem \ref{cor0}, one can see that the statement of Theorem \ref{RC} is equivalent to that of  Lemma \ref{rcslop}. To prove Lemma \ref{rcslop}, we need the following criteria for rational connectedness.
\begin{prop}\label{rationc}
 Let $M$ be a compact projective manifold. Then the  following statements are equivalent:
\begin{enumerate}
\item[(a)] $M$ is  rationally connected;
\item[(b)] There exist a movable curve $C$ and a constant $\delta>0$, such that for any coherent analytic quotient sheaf $\mathcal Q$ of $T^{1,0}M$, $c_1(\mathcal Q)\cdot[C]\geq\delta \cdot\rank(Q)$;
\item[(c)] For any $1\leq p\leq  \dim ^{\mathbb{C}} M$, any invertible subsheaf $\mathcal F\subset \Lambda^{p, 0}M$ is not pseudoeffective.
\end{enumerate}
\end{prop}

\begin{rem}\label{rem401}
Criterion $(c)$ was first given in \cite[Criterion 1.1]{CDP14}, together with the other two criteria.
$(b)\Rightarrow (a)$ was likely first shown as  a special case of \cite[Theorem 1.1]{CP2}. While $(a)\Leftrightarrow (b)$ was likely first appeared in \cite[Proposition 1.4]{Cam16}, which has a generalization for the orbifold case (\cite[Theorem 1.1]{Cam16}). In fact, even without using any results proved in \cite{CP2} or \cite{Cam16}, one can easily prove $(a)\Rightarrow (b)$, by using the following well known criterion of rational connectedness (see \cite[Theorem 3.7]{Kol96}):
\begin{itemize}
\item[$(b')$] $M$ is projective and there is a rational curve $f:\mathbb CP^1\rightarrow M$ such that $f^*T^{1,0}M$ is ample, namely $f^*T^{1,0}M=\oplus_{i=1}^{n} \mathcal O_{\mathbb CP^1}(a_i)$ for some positive integers $a_1\leq \cdots\leq a_n$.
\end{itemize}
Moreover, by the proof of $(a)\Rightarrow (b')$, $f$ in $(b')$ can be chosen to be a general member of an analytic family $\{f_t|t\in S\}$ of rational curves such that $\bigcup_{t\in S} f_t(\mathbb CP^1)=M$.

For the readers' convenience, we present a brief proof of $(a)\Rightarrow (b)$ here. Let $f$ be a rational curve in $(b')$ and $C=f(\mathbb CP^1)$, then clearly $C$ is movable and $f_\star[\mathbb CP^1]=m[C]$ for some positive integer $m$. We claim that if $\mathcal Q$ is a quotient sheaf of rank $p\geq 1$ of $T^{1,0}M$, then $c_1(\mathcal Q)\cdot[C]\geq \frac{p}{m}$. In fact, since $\det \mathcal Q=(\Lambda^p \mathcal Q)^{**}$, we have a natural morphism $\phi:\Lambda^p (T^{1,0}M)\rightarrow \det \mathcal Q$, which is surjective away from a proper analytic subset $Z$ of $M$. Up to replace $f$ by some sufficiently close $f_t$, we can assume that $C\not\subset Z$. Then $f^{-1}(Z)$ is empty or a proper analytic subset. Furthermore, the pullback of $\phi$, denoted by $\rho$, is surjective away from $f^{-1}(Z)$. Let $\mathcal L=\mathrm{Im} \rho$. Then $\mathcal L$ is a coherent  anlalytic sheaf of rank $1$. Moreover $\mathcal L$ is a quotient sheaf of $f^*T^{1,0}M$ as well as a subsheaf of $f^*\det \mathcal Q$. Noting that $f^*T^{1,0}M=\oplus_{i=1}^{n} \mathcal O_{\mathbb CP^1}(a_i)$ and $(f^*\det \mathcal Q)/\mathcal L$ is a torsion sheaf, we have
\begin{equation}
a_1+\cdots+a_p\leq c_1(\mathcal L)\cdot[\mathbb CP^1]\leq c_1(f^*\det\mathcal Q)\cdot[\mathbb CP^1]=mc_1(\mathcal Q)\cdot[C].
\end{equation}

\end{rem}

Here we review some notions and results related to the pseudo-effectiveness and movability.
Let $M$ be a compact complex manifold of dimension $n\geq 2$. We denote  $A^{p,q}(M,\mathbb C)$ to be the space of all smooth $(p,q)$-forms on $M$ and $A^{p,p}(M,\mathbb R)$ the space of all real smooth $(p,p)$-forms on $M$.
\begin{enumerate}
\item The real Bott-Chern cohomology group $H_{BC}^{p, p}(M,\mathbb R)$ is
\begin{equation}
H_{BC}^{p, p}(M,\mathbb R)=\frac{\{\theta\in  A^{p,p}(M,\mathbb R) \; | \; d\theta=0\}}{\sqrt{-1}\pbp A^{p-1,p-1}(M,\mathbb R)},
\end{equation}
and the real Aeppli cohomology group $H_{A}^{p, p}(M,\mathbb R)$ is
\begin{equation}
H_{A}^{p, p}(M,\mathbb R)=\frac{\{\theta\in  A^{p,p}(M,\mathbb R) \; | \; \sqrt{-1}\pbp\theta=0\}}{(\partial A^{p-1,p}(M, \mathbb C)+\bar\partial A^{p,p-1}(M, \mathbb C))\cap A^{p,p}(M,\mathbb R)}.
\end{equation}
The real Bott-Chern and Aeppli cohomology groups coincide with the corresponding cohomology groups of currents, respectively. Thus we use both smooth forms and currents as the representatives.
\item The Balanced cone is the open convex cone in $H_{BC}^{n-1,n-1}(M,\mathbb R)$
\begin{equation}
\mathcal B=\{[\omega^{n-1}]\in H_{BC}^{n-1,n-1}(M,\mathbb R) \; | \; \omega \text{ is a Balanced metric}\},
\end{equation}
and the Gauduchon cone is the open cone in $H_{A}^{n-1,n-1}(M,\mathbb R)$
\begin{equation}
\mathcal G=\{[\omega^{n-1}]\in H_{A}^{n-1,n-1}(M,\mathbb R) \; | \; \omega \text{ is a Gauduchon metric}\}.
\end{equation}
\item The pseudoeffective cone is the closed  convex cone in $H_{BC}^{1,1}(M,\mathbb R)$
\begin{equation}
\mathcal E=\{\alpha\in H_{BC}^{1,1}(M,\mathbb R) \; | \; \exists \text{ a positive $(1,1)$-current } T\in \alpha\}.
\end{equation}
\item Assume that $M$ is K\"ahler. The movable cone $\mathcal M\subset H^{n-1,n-1}_{BC}(M,\mathbb R)$ is the closure of the convex cone generated by currents of the form
\begin{equation}
\mu_\star(\tilde\omega_1\wedge\cdots\wedge\tilde\omega_{n-1}),
\end{equation}
where $\mu:\tilde M\rightarrow M$ is an arbitrary modification and $\tilde\omega_1,\cdots,\tilde \omega_{n-1}$ are K\"ahler forms on $\tilde M$.
\item Assume that $M$ is projective. A curve $C$ in $M$ is said to be movable if $C$ belongs to an analytic family $\{C_t \; | \; t\in S\}$  of curves in $M$ such that $\bigcup_{t\in S} C_t=M$.
Let ${\rm ME(M)}$ be the convex cone in
\begin{equation}
N_1(M)=H^{n-1,n-1}_{BC}(M,\mathbb R)\cap (H^{2n-2}(M,\mathbb Z)\otimes_{\mathbb Z} \mathbb R)
\end{equation}
generated by all movable curves. In fact,  \cite[Theorem 2.4]{BDPP13} states
\begin{equation}
\overline{{\rm ME(M)}}=\mathcal M\cap N_1(M).
\end{equation}
\end{enumerate}

Generally, we have the Poincar\'e  duality pairing
\begin{equation}\label{pd1}
H_{BC}^{p,p}(M,\mathbb R)\times H_A^{n-p,n-p}(M,\mathbb R)\rightarrow \mathbb R,\qquad (\alpha,\beta)\mapsto \int_M\alpha\wedge \beta.
\end{equation}
When $M$ is K\"ahler,  the Poincar\'e duality pairing can be written as
\begin{equation}\label{pd2}
H_{BC}^{p,p}(M,\mathbb R)\times H_{BC}^{n-p,n-p}(M,\mathbb R)\rightarrow \mathbb R,\qquad (\alpha,\beta)\mapsto \int_M\alpha\wedge \beta.
\end{equation}

Some useful criteria for pseudo-effectiveness are summarised below:
\begin{prop}[Dual of the pseudoeffective cone]\label{dualpef}
\item[(1)] Under the Poincar\'e duality pairing \eqref{pd1},
the dual $\mathcal E^{\vee}$ of $\mathcal E$ is  equal to the closure $\overline{\mathcal G}$ of the Gauduchon cone;
\item[(2)] If $M$ is K\"ahler, then with respect to the Poincar\'e duality pairing \eqref{pd2},
the cone duality $\mathcal E^{\vee}=\overline{\mathcal B}$ holds;
\item[(3)] If $M$ is projective, then the cones  $\mathcal E$ and $\mathcal M$ are dual via the Poincar\'e duality pairing \eqref{pd2}. Consequently $\mathcal M=\overline{\mathcal B}$.
\end{prop}
Indeed, $(1)$ is due to Lamari (\cite[Lemma 3.3]{Lam99}, see also \cite[Lemma 2.1]{Tos16}); $(2)$ is due to Fu-Xiao (\cite[Remark 3.4 and Theorem A.2]{FX14}) for the general K\"ahler case; $(3)$ is due to Nystr\"{o}m (\cite[Theorem A and Corollary A]{Nys19}.  Toma (\cite[Theorem]{Tom10}) also showed $\overline{\rm ME(M)}\subset \overline{\mathcal B}$ when $M$ is projective.

A holomorphic line bundle $L$ over a compact complex manifold $M$ is said to be pseudoeffective if its Chern class $c_1(L)$ is pseudoeffective. According to Propostion \ref{dualpef}, $\alpha\in H^{1,1}_{BC}(M,\mathbb R)$ is pseudoeffictive if and only if $\alpha\cdot[\omega^n]\geq 0$ for any Gauduchon metric $\omega$ on $M$. This result has been noticed by Yang \cite[Proposition 3.1 or 3.2]{Yang19}. Certainly it follows that
\begin{cor}
Let $L$ be a holomorphic line bundle over a compact complex manifold $M$. Then $L$ is not pseudoeffective if and only if there exists a Gauduchon metirc $\omega$ on $M$ such that $\deg_{\omega}(L)<0$.
\end{cor}


\begin{proof}[Proof of Lemma \ref{rcslop}]
Write $n= \dim ^{\mathbb{C}}M$. Assume that $(2)$ holds. By Theorem \ref{cor0} and Kobayashi-Wu's vanishing theorem (\cite{KobW70}), we know that $\Lambda^{p,0}M$ is mean curvature negative and then $H_{\bar{\partial } }^{p,0}(M)=0$ for every $1\leq p \leq n$. Since $M$ is K\"ahler, the projectivity follows from $H_{\bar{\partial } }^{2,0}(M)=H_{\bar{\partial } }^{0,2}(M)=0$ and the Kodaira theorem (\cite[Theorem 1]{Kod}). Let $L\subset\Lambda^{p,0}M$  be an invertible subsheaf. Then
\begin{equation}
\deg_\omega(L)\leq \mu_U(\Lambda^{p,0}M,\omega)\leq -p\mu_L(T^{1,0}M,\omega)<0,
\end{equation}
and consequently $L$ is not pseudoeffective. So we have $(2)\Rightarrow (1)$.

Assume that $(1)$ holds. On account of Proposition \ref{rationc}, we can find a movable curve $C$ and $\delta>0$ such that $c_1(\mathcal Q)\cdot[C]\geq \delta \cdot\rank(\mathcal Q)$ for any quotient sheaf $\mathcal Q$ of $T^{1,0}M$. One can always choose a K\"ahler metric $\omega_0$ and a constant $\varepsilon>0$, such that
\begin{equation}
\varepsilon\mu_L(T^{1,0}M,\omega_0)\geq-{\textstyle\frac{\delta}{2(n-1)!}}.
\end{equation}
Then for any quotient sheaf $\mathcal Q$ of $T^{1,0}M$, we have
\begin{equation}\label{slptx}
\begin{aligned}
c_1(\mathcal Q)\cdot([C]+\varepsilon[\omega_0^{n-1}])&=c_1(\mathcal Q)\cdot[C]+(n-1)!\varepsilon\mu_{\omega_0}(\mathcal Q)\cdot\rank(\mathcal Q)\\
&\geq(\delta+ (n-1)!\varepsilon\mu_L(T^{1,0}M,\omega_0))\cdot\rank(\mathcal Q)\\
&\geq {\textstyle\frac{\delta}{2}}\cdot\rank(\mathcal Q).
\end{aligned}
\end{equation}
The movability of $C$ means $[C]\in\mathcal M$. Note that Proposition \ref{dualpef} says $\mathcal M=\overline{\mathcal B}$. Thus $[C]\in \overline{\mathcal B}$. Since $[\omega_0^{n-1}]\in\mathcal B$, we observe $[C]+\varepsilon [\omega_0^{n-1}]\in \mathcal B$. Namely there is a balanced metric $\omega$ such that $[\omega^{n-1}]=[C]+\varepsilon [\omega_0^{n-1}]$. By (\ref{slptx}), it is evident that
\begin{equation}
\mu_L(T^{1,0}M,\omega)\geq {\textstyle\frac{\delta}{2(n-1)!}}.
\end{equation}
Therefore, we conclude $(1)\Rightarrow (2)$.
\end{proof}

In the following, we will prove that the uniformly RC-positivity  implies the mean curvature positivity.

\begin{prop}\label{RC01}
Let $(E, \bar{\partial}_{E})$ be a rank $r$ holomorphic vector bundle over an $n$-dimensional compact complex manifold $M$, and $H$ be a Hermitian metric on $E$. If $H$ is uniformly RC-positive, then it must be mean curvature positive, i.e. there exists a Hermitian metric $\omega $ on $M$ such that
\begin{equation}
\sqrt{-1}\Lambda_{\omega }F_{H}>0.
\end{equation}
\end{prop}

\begin{proof}
For every non-zero vector $v\in T_{x}^{1,0}M$,  $x\in M$, we know that $ F_{H}(v, \bar{v})\in \Gamma(\End (E))$ is $H$-selfadjoint. Explicitly  this tells us that all the eigenvalues of $ F_{H}(v, \bar{v})$ are real. Given a Hermitian metric $\omega_{0}$ on $M$, we set
\begin{equation}
\mu_{H, \omega_{0}, x}:= \sup_{v\in T_{x}^{1,0}M\setminus \{0\}}\lambda_{min}(\frac{F_{H}(v, \bar{v})}{-\sqrt{-1}\omega_{0}(v, \bar{v})}),
\end{equation}
\begin{equation}
\mu_{H, \omega_{0}}:=\min_{x\in M} \mu_{H, \omega_{0}, x}
\end{equation}
and
\begin{equation}
\nu_{H, \omega_{0}}:=\min_{x\in M} \inf_{v\in T_{x}^{1,0}M\setminus \{0\}}\lambda_{min}(\frac{F_{H}(v, \bar{v})}{-\sqrt{-1}\omega_{0}(v, \bar{v})}),
\end{equation}
where $\lambda_{min}$ stands for the smallest eigenvalue. Under the assumption that $H$ is uniformly RC-positive, there holds that
\begin{equation}
\mu_{H, \omega_{0}, x}\geq \mu_{H, \omega_{0}}>0.
\end{equation}

For any point $x\in M$,  choose a local $\omega_{0}$-orthonormal frame $\{e_{\alpha }\}_{\alpha =1}^{n}$ of $T^{1, 0}M$ around $x$ such that $\lambda_{min }(F_{H}(e_{1}(x), \overline{e_{1}(x)}))=\mu_{H, \omega_{0}, x}$. Let $\{\theta^{\alpha }\}_{\alpha =1}^{n}$ be the dual frame of $\{e_{\alpha }\}_{\alpha =1}^{n}$ and $a$ be a positive number. We construct a local Hermitian metric $\omega_{x, a}$ on $M$ by
\begin{equation}
\omega_{x, a}=\sqrt{-1}(\theta^{1}\wedge \overline{\theta^{1}}+\sum_{\alpha =2}^{n}a^{-1}\cdot \theta^{\alpha }\wedge \overline{\theta^{\alpha }}).
\end{equation}
If $0< a <\frac{1}{2(n-1)}\frac{\mu_{H, \omega_{0}}}{\max \{-\nu_{H, \omega_{0}}, 0\}}$, then
\begin{equation}\begin{aligned}
\sqrt{-1}\Lambda_{\omega_{x, a}}F_{H}(x)=& F_{H}(e_{1}(x), \overline{e_{1}(x)}) +\sum_{\alpha =2}^{n}a\cdot F_{H}(e_{\alpha }(x), \overline{e_{\alpha }(x)}) \\
\geq &  (\mu_{H, \omega_{0}} +(n-1)a\nu_{H, \omega_{0}})\Id_{E}\\ > & \frac{1}{2}\mu_{H, \omega_{0}}\Id_{E}.\\
\end{aligned}
\end{equation}
Hence for each point $x\in M$, one can find a neighborhood $B_{x}$ centered at $x$ and a Hermitian metric $\omega_{x, a}$  such that
\begin{equation}\label{oo1}
\sqrt{-1}\Lambda_{\omega_{x, a}}F_{H}>\frac{1}{4}\mu_{H, \omega_{0}}\Id_{E}
\end{equation}
on $B_{x}$. On the other hand, it is a simple matter to verify that
\begin{equation}
\omega_{x, a}^{n}=a^{-(n-1)}\omega_{0}^{n},
\end{equation}
and then the inequality (\ref{oo1}) is equivalent to
\begin{equation}\label{oo2}
\sqrt{-1}F_{H}\wedge \frac{(\omega_{x, a})^{n-1}}{(n-1)!}>\frac{1}{4}a^{-(n-1)}\mu_{H, \omega_{0}}\Id_{E} \frac{(\omega_{0})^{n}}{n!}.
\end{equation}

Because $M$ is compact, we can choose a finite open covering $\{B_{x_{i}}\}_{i=1}^{N}$ and a partition of unity $\{f_{i}\}_{i=1}^{N}$ subordinate to $\{B_{x_{i}}\}_{i=1}^{N}$, where $N$ is a finite positive integer. Set
\begin{equation}
\eta_{a}:=\sum_{i=1}^{N}f_{i}\cdot \omega_{x_{i}, a}^{n-1},
\end{equation}
and note  that $\eta_{a}$ is a strictly positive smooth $(n-1, n-1)$-form. Then there exists a unique Hermitian metric $\omega_{a}$ on $M$ (\cite{Mi}, p279) such that
\begin{equation}
\omega_{a}^{n-1}=\eta_{a}.
\end{equation}
Combining this with  (\ref{oo2}) yields
\begin{equation}\label{oo3}
\begin{aligned}
\sqrt{-1}F_{H}\wedge \frac{(\omega_{a})^{n-1}}{(n-1)!}&=\sqrt{-1}F_{H}\wedge \sum_{i=1}^{N}f_{i}\cdot \frac{(\omega_{x_{i}, a})^{n-1}}{(n-1)!}\\&>\frac{1}{4}a^{-(n-1)}\mu_{H, \omega_{0}}\Id_{E} \frac{(\omega_{0})^{n}}{n!},
\end{aligned}
\end{equation}
and then
\begin{equation}
\sqrt{-1}\Lambda_{\omega_{a}}F_{H}>0,
\end{equation}
which finishes the proof of Proposition \ref{RC01}.
\end{proof}

\begin{rem}
When $(M, \omega )$ is a compact K\"ahler manifold with positive holomorphic sectional curvatures, by Lemma 6.1 in \cite{Ya0} (see \cite{Lip} for compact Chern-K\"ahler-like Hermitian manifolds), we know that $(T^{1,0}M, \omega )$ is uniformly RC-positive, and then $T^{1,0}M$ is mean curvature positive.
\end{rem}

\section{Some applications}

\subsection{Calculating the minimal and maximal slopes}
Theorem \ref{minslope} provides a new way to calculate the minimal and maximal slopes in the Harder-Narasimhan types of tensor products, symmetric and exterior powers of holomorphic vector bundles. For instance, we have
\begin{thm}\label{thm5.1}
Let $(E, \bar{\partial}_{E})$ and  $(\tilde{E}, \bar{\partial}_{\tilde{E}})$ be two holomorphic vector bundles over a compact Gauduchon manifold $(M, \omega )$.  Then for $k,l\geq 0$, we have
\begin{align}
\label{tensorslope1}&\mu_L(E^{\otimes k}\otimes \tilde E^{\otimes l},\omega)=k\mu_L(E,\omega)+l\mu_L(\tilde E,\omega),\\
\label{tensorslope2}&\mu_U(E^{\otimes k}\otimes \tilde E^{\otimes l},\omega)=k\mu_U(E,\omega)+l\mu_U(\tilde E,\omega),\\
\label{symslope1}&\mu_{L}(S^kE, \omega )=k\mu_{L}(E, \omega ),\\
\label{symslope2}&\mu_{U}(S^kE, \omega )=k\mu_{U}(E, \omega ),
\end{align}
and for $1\leq k\leq\rank E$, we have
\begin{align}
\label{wedgeslope1}&\mu_L(\wedge^k E,\omega)\geq k\mu_L(E,\omega),\\
\label{wedgeslope2}&\mu_U(\wedge^k E,\omega)\leq k\mu_U(E,\omega).
\end{align}
\end{thm}

There are other practical ideas which work for proving Theorem \ref{thm5.1}.

\begin{proof}
We only need to prove \eqref{tensorslope1} for $k=l=1$, \eqref{symslope1} and \eqref{wedgeslope1}.

For convenience, we write for short
\begin{equation}
a=\frac{2\pi}{\Vol(M,\omega)}\mu_{L}(E,\omega),\qquad \tilde a=\frac{2\pi}{\Vol(M,\omega)}\mu_{L}(\tilde E,\omega).
\end{equation}
By Theorem \ref{minslope}, for any $\delta>0$, we can find  Hermitian metrics $H_{\delta}$ on $E$ and $\tilde H_{\delta}$ on $\tilde E$ such that
\begin{equation}
\sqrt{-1}\Lambda_{\omega}F_{H_{\delta}}\geq (a-\delta)\Id_{E},\qquad \sqrt{-1}\Lambda_{\omega}F_{\tilde H_{\delta}}\geq (\tilde a-\delta)\Id_{\tilde E}.
\end{equation}
Computing the mean curvatures of the induced Hermitian metrics $H_{\delta}\otimes \tilde H_{\delta}$ on $E\otimes \tilde E$, $S^kH_{\delta}$ on $S^kE$ $(k\geq 1)$ and $\wedge H^k$ on $\wedge ^kE$ ($1\leq k\leq \rank E$), we have
\begin{align}
&\sqrt{-1}\Lambda_{\omega}F_{H_{\delta}\otimes \tilde H_{\delta}}\geq (a+\tilde a-2\delta)\Id_{E\otimes \tilde E},\\
&\sqrt{-1}\Lambda_{\omega}F_{S^kH_{\delta}}\geq k(a-\delta)\Id_{S^kE},
\end{align}
where $k\geq 1$, and
\begin{equation}
\sqrt{-1}\Lambda_{\omega}F_{\wedge^kH_{\delta}}\geq k(a-\delta)\Id_{\wedge^kE},
\end{equation}
where $1\leq k\leq \rank E$. Applying Theorem \ref{minslope} again, we arrive at the ``$\geq$" parts.

Next we prove the ``$\leq$" parts. Notice that $E$ and $\tilde E$ have torsion free quotient sheaves $\mathcal Q_1$ and $\mathcal Q_2$ of positive ranks respectively, such that
\begin{equation}
\mu_{\omega}(\mathcal Q_1)=\mu_L(E,\omega),\qquad \mu_{\omega}(\mathcal Q_2)=\mu_L(\tilde E,\omega).
\end{equation}
Since $\mathcal Q_1\otimes \mathcal Q_2$ and $S^k\mathcal Q_1$ are quotient sheaves of $E\otimes \tilde E$ and $S^kE$ respectively, we reach
\begin{align}
&\mu_L(E\otimes\tilde E,\omega)\leq \mu_{\omega}(\mathcal Q_1\otimes\mathcal Q_2)\leq \mu_L(E,\omega)+\mu_L(\tilde E,\omega),\\
&\mu_L(S^kE)\leq \mu_\omega(S^k\mathcal Q_1)\leq k\mu_\omega(\mathcal Q_1).
\end{align}
From the definition of the minimal slope $\mu_L(\cdot,\omega)$, the ``$\leq$" parts come.
\end{proof}

\begin{cor}
The following statements are equivalent:
\begin{enumerate}
\item $\mu_L(E,\omega)>0$;
\item $\mu_L(E^{\otimes k},\omega)>0$ for some (resp. every) $k\geq 1$;
\item $\mu_L(S^k E,\omega)>0$ for some (resp. every) $k\geq 1$.
\end{enumerate}
Whenever one of the above holds, $\mu_L(\wedge^k E,\omega)>0$ for $1\leq k \leq\rank E$.
\end{cor}

\begin{cor}\label{vnshthm}
Let $(E, \bar{\partial}_{E})$ and  $(\tilde{E}, \bar{\partial}_{\tilde{E}})$ be two holomorphic vector bundles over a compact Gauduchon manifold $(M, \omega )$.  If $k\geq 0$, $l\geq 0$ and
\begin{equation}
k \mu_U(E,\omega) +l \mu_U(\tilde E,\omega)<0,
\end{equation}
then $\mu_U(E^{\otimes k}\otimes \tilde E^{\otimes l},\omega)<0$. Consequently
\begin{equation}\label{vnsh}
H^0(M, E^{\otimes k}\otimes \tilde E^{\otimes l})=0.
\end{equation}
\end{cor}

In Corollary \ref{vnshthm}, (\ref{vnsh}) is derived from Theorem \ref{cor0} and Kobayashi-Wu's vanishing theorem (\cite{KobW70}).

\medskip

\begin{proof}[Proof of Theorem \ref{cor111}]
Since $(E, \bar{\partial }_E)$ is HN-negative, we know that for any $\tau >0$, there exists a Gauduchon metric $\omega_{\tau}$ on $M$ such that $\mu_U(E, \omega_{\tau} )<0$ and
\begin{equation}
 G(M, E, \tilde{E})\leq \frac{-\mu_U(\tilde{E}, \omega_{\tau} )}{\mu_U(E, \omega_{\tau} )}< G(M, E, \tilde{E})+\tau .
\end{equation}
If $k>G(M, E, \tilde{E})l$, by choosing $\tau $ small enough, one can see
\begin{equation}\label{condi011}
k \mu_U(E, \omega_{\tau} )+l \mu_U(\tilde{E}, \omega_{\tau } )<0.
\end{equation}
According to  Corollary \ref{vnshthm}, we obtain $H^{0}(M, E^{\otimes k}\otimes \tilde E^{\otimes l})=0$.
\end{proof}

If $(E, \bar{\partial }_{E})$ is the holomorphic cotangent bundle $\Lambda^{1,0}M$ and $(\tilde{E}, \bar{\partial}_{\tilde{E}})$ is the holomorphic tangent bundle $T^{1,0}M$, we shall have established the corollary below.

\begin{cor}\label{corv1}
Let $M$ be a compact complex manifold. If $T^{1, 0}M$ is HN-positive, then
    \begin{equation}
    H^{0}(M, (T^{1,0}M)^{\otimes q}\otimes (\Lambda^{1,0}M)^{\otimes p})=0,
    \end{equation}
    when $p\geq 1$, $q\geq 0$ and $p>G(M, \Lambda^{1,0}M, T^{1, 0}M)q$.
\end{cor}

\medskip

\subsection{A proof of Theorem \ref{ample}}

Let $(E, \bar{\partial}_{E})$ be a rank $r$ holomorphic vector bundle over a compact complex manifold $M$. By \cite{Har66} and \cite{Gri69}, we have
\begin{enumerate}
\item If $E$ is ample, then $\Lambda^k E$ is ample for $1\leq k\leq r$;
\item If $E$ is ample, then any quotient bundle $Q$ of $E$ is ample;
\item \label{GPAMP} If $E$ is Griffiths positive, then $E$ is ample;
\item If $E$ is ample, then when $k$ is sufficiently large, $S^k E$ is Griffiths positive.
\end{enumerate}

Of course any quotient bundle of an ample bundle has positive first Chern class. However, it is not clear whether this property still holds in the quotient sheaf case. Fortunately we can confirm that there exists a K\"ahler current in the first Chern class.

\begin{prop}\label{kahlercurrent}
Let $\omega$ be a K\"ahler metric on $M$. Assume that $E$ is ample and $1 \leq p \leq r-1$. Then we can find $\delta_p>0$, such that for any  $p$-rank coherent quotient sheaf $\mathcal Q$ of $E$, there exists a current $\theta\in c_1(\mathcal Q)$ satisfying $\theta\geq\delta_p \omega$ in the sense of current.
\end{prop}

\begin{proof}Our idea of the proof originates from  the proof of \cite[Theorem 1.18]{DPS94}.
Because $\mathcal Q$ is a $p$-rank coherent quotient sheaf of $E$, $\mathcal Q^*$ is a $p$-rank coherent subsheaf of $E^*$. Then there is the following injective sheaf morphism
\begin{equation}
j:\det (\mathcal Q^*)\rightarrow \Lambda^p E^*.
\end{equation}
Passing to symmetric powers, for $k \geq 1$, we have  the injective sheaf morphism
\begin{equation}
j_k:(\det(\mathcal Q^*))^k\rightarrow S^k(\Lambda^p E^*).
\end{equation}
Thanks to these sheaf morphisms, one can construct a singular Hermitian metric whose ``Chern curvature"  is a K\"ahler current on $\det \mathcal Q$ .

Since $\Lambda^p E$ is also ample, we can find $k_0\geq 1$, $a>0$ and Hermitian metric $\hat H$ on $S^{k_0}(\Lambda^p E)$ such that
\begin{equation}
\sqrt{-1} F_{\hat H}\geq a\Id\otimes\omega
\end{equation}
in the sense of Griffiths. Suppose $\tilde H$ is the induced Hermitian metric on $S^{k_0}(\Lambda^p E^*)=(S^{k_0}(\Lambda^p E))^*$ by $\hat{H}$, then
\begin{equation}
\sqrt{-1} F_{\tilde H}\leq -a\Id\otimes\omega
\end{equation}
in the sense of Griffiths.

Let $h$ be a smooth Hermitian metric on $\det \mathcal Q$. We define $\varphi:M\rightarrow [-\infty,\infty)$ as
\begin{equation}\label{defphi}
e^{k_0\varphi(z)}=\frac{|j_{k_0}(\xi)|^2_{\tilde H}}{|\xi|^2_{h^{-k_0}}},
\end{equation}
where $z\in M$ and $\xi$ is an arbitrary non-zero element of $(\det(\mathcal Q^*))^{k_0}|_z$.
Replacing $\xi$ by nowhere vanishing  local holomorphic sections of $(\det(\mathcal Q^*))^{k_0}$, we can get the local expressions of $\varphi$. Based on the local expressions of $\varphi$, one can see at once that $\varphi\in L^1(M)$ and consequently $\sqrt{-1} F_h+\sqrt{-1} \pbp \varphi$ is a well-defined $(1,1)$-current.
Furthermore, consider
\begin{equation}
Z=\{ z\in M \ | \ j|_{\det(\mathcal Q^*)|_z}:\det(\mathcal Q^*)|_z\rightarrow \Lambda^p E^* |_z\text{ is not injective}\},
\end{equation}
then over $M\setminus Z$, $j_{k_0}:(\det(\mathcal Q^*))^{k_0}\rightarrow(S^{k_0}\Lambda^p E)^*$ is a subbundle and $h^{-k_0}e^{k_0\varphi}$ is actually the induced Hermitian metric by $\tilde H$. If $u$ is a nowhere vanishing holomorphic section of $(\det(\mathcal Q^*))^{k_0}$ on some open subset of $M\!\setminus\! Z$, set $s=j_{k_0}(u)$ and $\tilde s=|s|_{\tilde H}^{-1}s$. Then by virtue of the Gauss-Codazzi equation for subbundles, we have
\begin{equation}\begin{aligned}
-\sqrt{-1} k_0(F_h+\pbp\varphi)&=\sqrt{-1} F_{h^{-k_0}e^{k_0\varphi}}\\
& =\sqrt{-1} \langle F_{\tilde H} \tilde s,\tilde s\rangle_{\tilde H}-\sqrt{-1} \langle\beta \tilde s,\beta \tilde s\rangle_{\tilde H}\\
&\leq \sqrt{-1} \langle F_{\tilde H} \tilde s,\tilde s\rangle_{\tilde H},
\end{aligned}\end{equation}
where $\beta$ is the $(1,0)$-component of the second fundamental form.
One can directly verify that
\begin{equation}
\sqrt{-1}(F_h+\pbp\varphi)\geq \frac{a}{k_0}\omega
\end{equation}
on $M \!\setminus\! Z$. Note that $\lim_{z'\rightarrow z}\varphi(z')=-\infty$ for any $z\in Z$. It is easy to check that actually $\sqrt{-1}(F_h+\pbp\varphi)\geq \frac{a}{k_0}\omega$ in the sense of current on the whole of $M$.
The fact that $\sqrt{-1}(F_h+\pbp\varphi)\in 2\pi c_1(\det \mathcal Q)$ finishes this proof.
\end{proof}

As a simple corollary of Proposition \ref{kahlercurrent}, we infer

\begin{cor}\label{ample2}
If $E$ is ample, then we can find a K\"ahler metric $\omega_0$ on $M$, such that for any Gauduchon metric $\omega$, we have
\begin{equation}
\mu_{L}(E, \omega )\geq \int_M\omega_0\wedge\frac{\omega^{n-1}}{(n-1)!}.
\end{equation}
\end{cor}

\begin{proof}[Proof of Theorem \ref{ample}] Apply Corollary \ref{ample2} and Theorem \ref{cor0}.\end{proof}

\subsection{A integral inequality for holomorphic maps}

Let $(M, \omega )$ be a Hermitian manifold of complex dimension $m$. In the local complex coordinate $\{z^\alpha\}_{\alpha=1}^m$, the K\"ahler form $\omega $ and the curvature tensor $F_{\omega }$ of the Chern connection $D_{\omega }$ can be expressed as
\begin{equation}
\omega =\sqrt{-1}g_{\alpha \bar{\beta }} dz^{\alpha }\wedge d\bar{z}^{\beta},
\end{equation}
\begin{equation}
F_{\omega }(\frac{\partial }{\partial z^{\alpha }}, \frac{\partial }{\partial \bar{z}^{\beta }})\frac{\partial }{\partial z^{\gamma }}=(F_{\omega })_{\alpha \bar{\beta } \gamma}^{\eta}\frac{\partial }{\partial z^{\eta }}
\end{equation}
and
\begin{equation}
(F_{\omega })_{\alpha \bar{\beta } \gamma}^{\eta}=-g^{\eta \bar{\xi}}\frac{\partial^{2} g_{\gamma \bar{\xi}}}{\partial z^{\alpha }\bar{z}^{\beta }}+
g^{\eta \bar{\tau}}g^{\zeta \bar{\xi}}\frac{\partial g_{\gamma \bar{\xi}}}{\partial z^{\alpha }}\frac{\partial g_{\zeta \bar{\tau}}}{\partial \bar{z}^{\beta }},
\end{equation}
where
$(g^{\alpha\bar{\beta}})$ is the transpose of the inverse matrix of $(g_{\alpha\bar{\beta}})$. For any $X, Y \in T_{x}^{1, 0}(M)\setminus \{0\}$, $x \in M$, the holomorphic bisectional curvature is defined by
\begin{equation}
HB^{\omega}_{x}(X, Y)=\frac{\langle F_{\omega }(X, \bar{X})Y, Y\rangle_{\omega }}{|X|_{\omega}^{2}|Y|_{\omega }^{2}},
\end{equation}
where $\langle \cdot , \cdot \rangle_{\omega}$ is the Hermitian inner product induced by $\omega$. The supremum of holomorphic bisectional curvatures at $x\in M$ is given by
\begin{equation}
HB^{\omega}_{x}:=\sup \{HB^{\omega}_{x}(X, Y)\ | \ X, Y \in T_{x}^{1, 0}(M)\setminus \{0\}\}.
\end{equation}

\begin{prop}\label{bochner}
Let $f$ be a holomorphic map from a Gauduchon manifold $(M, \omega )$ to a Hermitian manifold $(N, \nu )$.  Then for any Hermitian metric $H$ on $T^{1, 0}M$, there holds
\begin{equation}\label{EQN6}
\begin{split}
\sqrt{-1}\Lambda_{\omega}\partial \bar{\partial}|\partial f|_{H, \nu}^2\geq& g^{\alpha\bar{\beta}}\big\langle \nabla_{\tfrac{\partial}{\partial z^\alpha}}\partial f, \nabla_{\tfrac{\partial}{\partial z^{\beta}}}\partial f \big\rangle_{H, \nu}\\
&+ \lambda_L(H, \omega)|\partial f|_{H, \nu}^2- HB^{\nu}_{f(\cdot)}|\partial f|_{H, \nu}^2|\partial f|_{\omega, \nu}^2,
\end{split}
\end{equation}
where $\nabla$ is the connection on $\Lambda^{1, 0}M\otimes f^{\ast}(T^{1, 0}N)$ induced by the Chern connection $D_{H}$  on $T^{1, 0}M$ and the Chern connection $D_\nu$ on $T^{1, 0}N$,  $HB^{\nu}_{f(\cdot)}$ is the supremum of holomorphic bisectional curvatures at $f(\cdot)\in (N, \nu)$, $|\partial f |_{H, \nu}$ and $|\partial f|_{\omega, \nu}$ are the norms of $\partial f$ as a section of $\Lambda^{1, 0}M\otimes f^{\ast}(T^{1, 0}N)$.
\end{prop}

\begin{proof}
Write $n= \dim ^{\mathbb{C}} N$. In the local complex coordinates $\{z^\alpha\}_{\alpha=1}^m$ on $M$ and $\{w^i\}_{i=1}^n$ on $N$, we set $\langle\frac{\partial}{\partial z^\alpha}, \frac{\partial}{\partial z^\beta}\rangle_H= \tilde{g}_{\alpha\bar{\beta}}$ and $\langle\frac{\partial}{\partial w^i}, \frac{\partial}{\partial w^j}\rangle_{\nu}= \nu_{i\bar{j}}$. Let $H^\ast$ be the Hermitian metric on $\Lambda^{1, 0}M$ induced by $H$ and $\langle d z^\alpha, d z^\beta\rangle_{H^\ast}= \tilde{g}^{\alpha\bar{\beta}}$. Then one has the following local expressions
\begin{equation}
\partial f = \frac{\partial f^i}{\partial z^\alpha}dz^\alpha\otimes \frac{\partial}{\partial w^i} \quad \text{and}\quad  |\partial f |^2_{H, \nu}= \frac{\partial f^i}{\partial z^\alpha}\Big(\overline{\frac{\partial f^j}{\partial z^\beta}}\Big)\tilde{g}^{\alpha\bar{\beta}}\nu_{i\bar{j}}.
\end{equation}
Let $g$ be the Hermitian metric whose associated $(1, 1)$-form is $\omega$. Write $g(\frac{\partial}{\partial z^\alpha}, \frac{\partial}{\partial z^\beta})= g_{\alpha\bar{\beta}}$. So
\begin{equation}\label{EQN1}
\sqrt{-1}\Lambda_{\omega}\partial\bar{\partial}|\partial f |^2_{H, \nu}=\frac{\sqrt{-1}\partial\bar{\partial}|\partial f |^2_{H, \nu}\wedge \frac{\omega^{m-1}}{(m-1)!}}{\frac{\omega^{m}}{m!}}= g^{\alpha\bar{\beta}}\frac{\partial^2}{\partial z^\alpha\partial \bar{z}^\beta} |\partial f |^2_{H, \nu}.
\end{equation}

 Moreover, the condition that $f$ is holomorphic gives us
\begin{equation}
\nabla_{\tfrac{\partial}{\partial \bar{z}^\beta}}\partial f=0.
\end{equation}
Clearly there is
\begin{equation}\label{EQN2}
\frac{\partial^2}{\partial z^\alpha\partial \bar{z}^\beta} |\partial f |^2_{H, \nu}=\big\langle \nabla_{\tfrac{\partial}{\partial z^\alpha}}\partial f, \nabla_{\tfrac{\partial}{\partial z^\beta}}\partial f\big\rangle_{H, \nu}+\big\langle \nabla_{\tfrac{\partial}{\partial \bar{z}^\beta}}\nabla_{\tfrac{\partial}{\partial z^\alpha}}\partial f, \partial f \big\rangle_{H, \nu}.
\end{equation}
A direct calculation yields that
\begin{equation}
\begin{split}
&\nabla_{\tfrac{\partial}{\partial \bar{z}^\beta}}\nabla_{\tfrac{\partial}{\partial z^\alpha}}\partial f\\
=& \big(\nabla_{\tfrac{\partial}{\partial \bar{z}^\beta}}\nabla_{\tfrac{\partial}{\partial z^\alpha}}-\nabla_{\tfrac{\partial}{\partial z^\alpha}}\nabla_{\tfrac{\partial}{\partial \bar{z}^\beta}}\big)\partial f\\
=& \frac{\partial f^i}{\partial z^\gamma} \big(\nabla_{\tfrac{\partial}{\partial \bar{z}^\beta}}\nabla_{\tfrac{\partial}{\partial z^\alpha}}-\nabla_{\tfrac{\partial}{\partial z^\alpha}}\nabla_{\tfrac{\partial}{\partial \bar{z}^\beta}}\big)(dz^\gamma\otimes\frac{\partial }{\partial w^i})\\
=& \frac{\partial f^i}{\partial z^\gamma}\big((F_{H^\ast}(\frac{\partial}{\partial \bar{z}^\beta}, \frac{\partial}{\partial z^\alpha})dz^\gamma)\otimes \frac{\partial }{\partial w^i}+ dz^\gamma\otimes (F_\nu(f_{\ast}(\frac{\partial}{\partial \bar{z}^\beta}), f_{\ast}(\frac{\partial}{\partial z^\alpha}))\frac{\partial }{\partial w^i}) \big),
\end{split}
\end{equation}
and then
\begin{equation}\label{EQN3}
\begin{split}
&g^{\alpha\bar{\beta}}\big\langle \nabla_{\tfrac{\partial}{\partial \bar{z}^\beta}}\nabla_{\tfrac{\partial}{\partial z^\alpha}}\partial f, \partial f \big\rangle_{H, \nu}\\
= &\big\langle -\frac{\partial f^i}{\partial z^\gamma}(\sqrt{-1}\Lambda_\omega F_{H^\ast}(dz^\gamma))\otimes \frac{\partial }{\partial w^i}, \partial f\big\rangle_{H, \nu}\\
& +\big\langle g^{\alpha\bar{\beta}}\frac{\partial f^i}{\partial z^\gamma}dz^\gamma\otimes (F_\nu(f_{\ast}(\frac{\partial}{\partial \bar{z}^\beta}), f_{\ast}(\frac{\partial}{\partial z^\alpha}))\frac{\partial }{\partial w^i}),  \partial f\big\rangle_{H, \nu}\\
= & -\big\langle \sqrt{-1}\Lambda_\omega F_{H^\ast}(dz^\gamma), dz^\xi\big\rangle_{H^\ast}\frac{\partial f^i}{\partial z^\gamma}\Big(\overline{\frac{\partial f^j}{\partial z^\xi}}\Big)\nu_{i\bar{j}}\\
&+ \frac{\partial f^i}{\partial z^\gamma}\Big(\overline{\frac{\partial f^j}{\partial z^\xi}}\Big)\tilde{g}^{\gamma\bar{\xi}}g^{\alpha\bar{\beta}}\big\langle F_\nu(f_{\ast}(\frac{\partial}{\partial \bar{z}^\beta}), f_{\ast}(\frac{\partial}{\partial z^\alpha}))(\frac{\partial }{\partial w^i}), \frac{\partial }{\partial w^j}\big\rangle_{\nu},
\end{split}
\end{equation}
where $F_{H^\ast}$ and $F_\nu$ are the curvatures of $D_{H^{\ast}}$ and $D_\nu$, respectively.

At the considered point $x\in M$, one can choose a local complex coordinate $\{z^1, \cdots, z^m\}$ centered at $x$ such that
\begin{equation}
g^{\alpha\bar{\beta}}(x)=\delta_{\alpha\beta} \quad \text{and} \quad \tilde{g}_{\gamma\bar{\xi}}(x)=a_\gamma \delta_{\gamma\xi},
\end{equation}
where for every $1\leq \gamma \leq m$, $a_\gamma$ is a positive number.
Notice that
\begin{equation}
f_{\ast}(\frac{\partial}{\partial z^\alpha})= \frac{\partial f^i}{\partial z^\alpha}\frac{\partial }{\partial w^i}.
\end{equation}
At $x$, we can write
\begin{equation}
|\partial f|_{\omega, \nu}^2= g^{\alpha\bar{\beta}}\big\langle \frac{\partial f^i}{\partial z^\alpha}\frac{\partial }{\partial w^i}, \frac{\partial f^j}{\partial z^\beta}\frac{\partial }{\partial w^j}\big\rangle_{\nu}= \sum_{\alpha=1}^m  |f_{\ast}(\frac{\partial}{\partial z^\alpha})|^2_{\nu}
\end{equation}
and
\begin{equation}
\begin{split}
|\partial f|_{H, \nu}^2=& \tilde{g}^{\gamma\bar{\xi}}\big\langle \frac{\partial f^i}{\partial z^\gamma}\frac{\partial }{\partial w^i}, \frac{\partial f^j}{\partial z^\xi}\frac{\partial }{\partial w^j}\big\rangle_{\nu}\\
=& \sum_{\gamma=1}^m \frac{1}{a_{\gamma}}\big\langle \frac{\partial f^i}{\partial z^\gamma}\frac{\partial }{\partial w^i}, \frac{\partial f^j}{\partial z^\gamma}\frac{\partial }{\partial w^j}\big\rangle_{\nu}\\
=& \sum_{\gamma=1}^m \frac{1}{a_{\gamma}}|f_{\ast}(\frac{\partial}{\partial z^\gamma})|^2_{\nu}.
\end{split}
\end{equation}
Furthermore, we compute
\begin{equation}\label{EQN4}
\begin{split}
&\frac{\partial f^i}{\partial z^\gamma}\Big(\overline{\frac{\partial f^j}{\partial z^\xi}}\Big)\tilde{g}^{\gamma\bar{\xi}}g^{\alpha\bar{\beta}}\big\langle F_\nu(f_{\ast}(\frac{\partial}{\partial \bar{z}^\beta}), f_{\ast}(\frac{\partial}{\partial z^\alpha}))(\frac{\partial }{\partial w^i}), \frac{\partial }{\partial w^j}\big\rangle_{\nu}(x)\\
=& \sum_{\gamma=1}^m a_{\gamma}^{-1}\frac{\partial f^i}{\partial z^\gamma}\Big(\overline{\frac{\partial f^j}{\partial z^\gamma}}\Big)\big\langle \sum_{\alpha=1}^m F_\nu(f_{\ast}(\frac{\partial}{\partial \bar{z}^\alpha}), f_{\ast}(\frac{\partial}{\partial z^\alpha}))(\frac{\partial }{\partial w^i}), \frac{\partial }{\partial w^j}\big\rangle_{\nu}(x)\\
=& \sum_{\gamma=1}^m \sum_{\alpha=1}^m \big\langle F_\nu(f_{\ast}(\frac{\partial}{\partial \bar{z}^\alpha}), f_{\ast}(\frac{\partial}{\partial z^\alpha}))(\sqrt{a_{\gamma}^{-1}}\frac{\partial f^i}{\partial z^\gamma}\frac{\partial }{\partial w^i}), \sqrt{a_{\gamma}^{-1}}\frac{\partial f^j}{\partial z^\gamma}\frac{\partial }{\partial w^j}\big\rangle_{\nu}(x)\\
=&- \sum_{\gamma=1}^m \sum_{\alpha=1}^m HB^{\nu}_{f(x)}(f_{\ast}(\frac{\partial}{\partial z^\alpha}), \sqrt{a_{\gamma}^{-1}}f_{\ast}(\frac{\partial}{\partial z^\gamma}))\big|f_{\ast}(\frac{\partial}{\partial z^\alpha})\big|^2_{\nu}\cdot\big|\sqrt{a_{\gamma}^{-1}}f_{\ast}(\frac{\partial}{\partial z^\gamma})\big|^2_{\nu}(x)\\
\geq & -HB^{\nu}_{f(x)}\sum_{\gamma=1}^m \sum_{\alpha=1}^m \big|f_{\ast}(\frac{\partial}{\partial z^\alpha})\big|^2_{\nu}\cdot\big|\sqrt{a_{\gamma}^{-1}}f_{\ast}(\frac{\partial}{\partial z^\gamma})\big|^2_{\nu}(x)\\
= & -HB^{\nu}_{f(x)}\big|\partial f\big|_{\omega, \nu}^2\cdot\big|\partial f\big|_{H, \nu}^2(x).
\end{split}
\end{equation}

On the other hand, the fact $\sqrt{-1}\Lambda_{\omega}F_H\geq \lambda_L(H, \omega)\Id$ implies $-\sqrt{-1}\Lambda_{\omega}F_{H^\ast}\geq \lambda_L(H, \omega)\Id$. Thus
\begin{equation}\label{EQN5}
-\langle \sqrt{-1}\Lambda_{\omega}F_{H^\ast}(dz^{\gamma}), dz^{\xi}\rangle_{H^\ast}\frac{\partial f^i}{\partial z^\gamma}\Big(\overline{\frac{\partial f^j}{\partial z^\xi}}\Big)\nu_{i\bar{j}}\geq \lambda_L(H, \omega)|\partial f|_{H, \nu}^2.
\end{equation}
This together with  (\ref{EQN1}), (\ref{EQN2}), (\ref{EQN3}) and (\ref{EQN4}) gives (\ref{EQN6}).

\end{proof}

\begin{thm}\label{holomorphic2}

Let $f$ be a holomorphic map from a compact Gauduchon manifold $(M, \omega )$ to a Hermitian manifold $(N, \nu )$.  If $f$ is not constant, then for any Hermitian metric $H$ on $T^{1, 0}M$, there holds
\begin{equation}\label{holomorphic4}
\int_{M}\lambda_L(H,\omega)\frac{\omega^{m}}{m!}\leq \int_{M}HB^{\nu }_{f(\cdot )}\cdot f^{\ast }(\nu ) \wedge \frac{\omega^{m-1}}{(m-1)!},
\end{equation}
where $m=\dim ^{\mathbb{C}}M$.

\end{thm}

\begin{proof}
By (\ref{EQN6}), we have
\begin{equation}\label{EQN7}
\begin{split}
&\sqrt{-1}\Lambda_{\omega}\partial \bar{\partial}\log(|\partial f|_{H, \nu}^2+ \varepsilon)
= \sqrt{-1}\Lambda_{\omega}\partial \Big(\frac{\bar{\partial} |\partial f|_{H, \nu}^2}{|\partial f|_{H, \nu}^2+ \varepsilon}\Big) \\
=&\frac{\sqrt{-1}\Lambda_{\omega}\partial \bar{\partial}|\partial f|_{H, \nu}^2}{|\partial f|_{H, \nu}^2+ \varepsilon}+ \frac{\sqrt{-1}\Lambda_{\omega}\bar{\partial} |\partial f|_{H, \nu}^2\wedge \partial |\partial f|_{H, \nu}^2}{(|\partial f|_{H, \nu}^2+ \varepsilon)^2}\\
\geq&\frac{g^{\alpha\bar{\beta}}\big\langle \nabla_{\tfrac{\partial}{\partial z^\alpha}}\partial f, \nabla_{\tfrac{\partial}{\partial z^{\beta}}}\partial f \big\rangle_{H, \nu}}{|\partial f|_{H, \nu}^2+ \varepsilon}- \frac{g^{\alpha\bar{\beta}}\frac{\partial }{\partial z^{\alpha}}|\partial f|_{H, \nu}^2\cdot\frac{\partial }{\partial \bar{z}^{\beta}}|\partial f|_{H, \nu}^2}{(|\partial f|_{H, \nu}^2+ \varepsilon)^2}\\
&+ \lambda_L(H, \omega)\frac{|\partial f|_{H, \nu}^2}{|\partial f|_{H, \nu}^2+ \varepsilon}- HB^{\nu}_{f(\cdot)}\frac{|\partial f|_{H, \nu}^2|\partial f|_{\omega, \nu}^2}{|\partial f|_{H, \nu}^2+ \varepsilon},
\end{split}
\end{equation}
where $\varepsilon > 0$ is small enough.
Choose the local complex coordinate $\{z^1, \cdots, z^m\}$ such that $g^{\alpha\bar{\beta}}=\delta_{\alpha\beta}$ at the considered point. Then
\begin{equation}
\begin{split}
&\frac{g^{\alpha\bar{\beta}}\big\langle \nabla_{\tfrac{\partial}{\partial z^\alpha}}\partial f, \nabla_{\tfrac{\partial}{\partial z^{\beta}}}\partial f \big\rangle_{H, \nu}}{|\partial f|_{H, \nu}^2+ \varepsilon}- \frac{g^{\alpha\bar{\beta}}\frac{\partial }{\partial z^{\alpha}}|\partial f|_{H, \nu}^2\cdot\frac{\partial }{\partial \bar{z}^{\beta}}|\partial f|_{H, \nu}^2}{(|\partial f|_{H, \nu}^2+ \varepsilon)^2}\\
=& \frac{1}{|\partial f|_{H, \nu}^2+ \varepsilon}\Big(\sum_{\alpha=1}^m \big|\nabla_{\tfrac{\partial}{\partial z^\alpha}}\partial f\big|_{H, \nu}^2-\frac{\sum_{\alpha=1}^m \big|\frac{\partial }{\partial z^{\alpha}}|\partial f|_{H, \nu}^2\big|^2}{|\partial f|_{H, \nu}^2+ \varepsilon}\Big)\\
\geq& \frac{\varepsilon}{(|\partial f|_{H, \nu}^2+ \varepsilon)^2}\big(\sum_{\alpha=1}^m \big|\nabla_{\tfrac{\partial}{\partial z^\alpha}}\partial f\big|_{H, \nu}^2\big)\\
=& \frac{\varepsilon}{(|\partial f|_{H, \nu}^2+ \varepsilon)^2}g^{\alpha\bar{\beta}}\big\langle \nabla_{\tfrac{\partial}{\partial z^\alpha}}\partial f, \nabla_{\tfrac{\partial}{\partial z^{\beta}}}\partial f \big\rangle_{H, \nu},
\end{split}
\end{equation}
where the inequality is due to
\begin{equation}
\big|\frac{\partial }{\partial z^{\alpha}}|\partial f|_{H, \nu}^2\big|= \big|\big\langle\nabla_{\tfrac{\partial}{\partial z^\alpha}}\partial f, \partial f\big\rangle_{H, \nu}\big|\leq  \big|\nabla_{\tfrac{\partial}{\partial z^\alpha}}\partial f\big|_{H, \nu}\cdot |\partial f|_{H, \nu}.
\end{equation}
Hence
\begin{equation}\label{EQN8}
\begin{split}
&\sqrt{-1}\Lambda_{\omega}\partial \bar{\partial}\log(|\partial f|_{H, \nu}^2+ \varepsilon)\\
\geq &\frac{\varepsilon}{(|\partial f|_{H, \nu}^2+ \varepsilon)^2}g^{\alpha\bar{\beta}}\big\langle \nabla_{\tfrac{\partial}{\partial z^\alpha}}\partial f, \nabla_{\tfrac{\partial}{\partial z^{\beta}}}\partial f \big\rangle_{H, \nu}\\
&+ \lambda_L(H, \omega)\frac{|\partial f|_{H, \nu}^2}{|\partial f|_{H, \nu}^2+ \varepsilon}- HB^{\nu}_{f(\cdot)}\frac{|\partial f|_{H, \nu}^2|\partial f|_{\omega, \nu}^2}{|\partial f|_{H, \nu}^2+ \varepsilon}\\
\geq& \frac{|\partial f|_{H, \nu}^2}{|\partial f|_{H, \nu}^2+ \varepsilon}(\lambda_L(H, \omega)-HB^{\nu}_{f(\cdot)}|\partial f|_{\omega, \nu}^2).
\end{split}
\end{equation}
Integrating (\ref{EQN8}) with respect to $\frac{\omega^m}{m!}$ over $M$, and noting that $\omega$ is Gauduchon, one has
\begin{equation}\label{EQN9}
\int_M \frac{|\partial f|_{H, \nu}^2}{|\partial f|_{H, \nu}^2+ \varepsilon}(\lambda_L(H, \omega)-HB^{\nu}_{f(\cdot)}|\partial f|_{\omega, \nu}^2)\frac{\omega^m}{m!}\leq 0.
\end{equation}

If $f$ is not constant, $\tilde{\Sigma}:= \{x\in M \ | \ \partial f (x)=0 \}$ is a proper subvariety of $M$. Applying Lebesgue's dominated convergence theorem, we deduce
\begin{equation}\label{EQN10}
\begin{split}
&\int_M (\lambda_L(H, \omega)-HB^{\nu}_{f(\cdot)}|\partial f|_{\omega, \nu}^2)\frac{\omega^m}{m!}\\
= &\lim_{\varepsilon\to 0}\int_M \frac{|\partial f|_{H, \nu}^2}{|\partial f|_{H, \nu}^2+ \varepsilon}(\lambda_L(H, \omega)-HB^{\nu}_{f(\cdot)}|\partial f|_{\omega, \nu}^2)\frac{\omega^m}{m!}\\
\leq &0.
 \end{split}
\end{equation}

\end{proof}

If $H$ is  the Hermitian metric on $T^{1, 0}(M)$ induced by $\omega $, then the mean curvature $\sqrt{-1}\Lambda _{\omega }F_{H}$ is just the second Chern-Ricci curvature of $\omega$. In this special case,  the inequality (\ref{holomorphic4}) was proved recently by Zhang (\cite{Zy}).

\begin{proof}[Proof of Theorem \ref{holomorphic1}]
Let $\{H_{\varepsilon_i}\}$ be a sequence of Hermitian metrics given in Theorem \ref{theorem1}. Combining (\ref{holomorphic4}) and Theorem \ref{theorem1}, we derive
\begin{equation}\label{holomorphic401}
2\pi \mu_{L}(T^{1, 0}M, \omega )=\lim_{i\rightarrow \infty } \int_{M}\lambda_L(H_{\varepsilon_i},\omega)\frac{\omega^{m}}{m!}\leq \int_{M}HB^{\nu }_{f(\cdot )}\cdot f^{\ast }(\nu ) \wedge \frac{\omega^{m-1}}{(m-1)!}.
\end{equation}
This concludes the proof of Theorem \ref{holomorphic1}.

\end{proof}

\bigskip

\medskip

{\bf Conflict of interest} We declare that we have no financial and personal relationships with other people or
organizations that can inappropriately influence our work, there is no professional or other personal interest
of any nature or kind in any product, service and company that could be construed as influencing the position
presented in, or the review of, this manuscript.

{\bf  Data availability statement} All data generated or analysed during this study are included in this published article.



\end{document}